\documentclass[final,leqno,onetabnum]{siamltex1213}
\usepackage{amsfonts,amssymb,amsmath}
\usepackage[firstinits=true,backend=bibtex,doi=false,isbn=false,url=false,eprint=false]{biblatex} 
\renewbibmacro{in:}{\ifentrytype{article}{}{\printtext{\bibstring{in}\intitlepunct}}} 
\usepackage{verbatim}
\usepackage{subcaption}
\usepackage{graphicx, color}
\usepackage[colorlinks]{hyperref}
\graphicspath{{./figures/}}  
\newcommand{\R}{\mathbb{R}}

\newcommand{\N}{\mathbb{N}}
\newcommand{\K}{\mathcal{K}}
\newcommand{\dd}{d}
\newcommand{\eqdef}{\overset{\text{def}}{=}} 
\newcommand{\COV}[1]{\mathbf{Cov} \left[ #1 \right]}
\providecommand{\Null}[1]{\mathbf{Null}\left( #1\right)}
\providecommand{\myRange}[1]{\mathbf{Range}\left( #1\right)}
\providecommand{\Rank}[1]{\mathbf{Rank}\left( #1\right)}
\providecommand{\Tr}[1]{\mathbf{Tr}\left( #1\right)}
\providecommand{\E}[1]{\mathbf{E}\left[ #1\right]}

\providecommand{\rob}[1]{#1}
\providecommand{\robdelete}[1]{}

\providecommand{\dotprod}[1]{\left< #1\right>} 

\providecommand{\norm}[1]{\| #1 \|}  
\newtheorem{assumption}{Assumption}[section]
\newcommand{\Ref}[1]{../../ref/#1}
\bibliography{\Ref{Krylov},\Ref{quasi-Newton_Conjugate_Gradients},\Ref{Stochastic_methods_models},\Ref{Image_and_signal_recovery},\Ref{Inverting_matrices},\Ref{Statistics_prob},\Ref{Optimization_methods},\Ref{Peter}}
\usepackage{tikz}
\usetikzlibrary{decorations.pathreplacing,calc,positioning,arrows,shapes}

\title{Randomized Iterative Methods for Linear Systems}
\author{Robert M.~Gower \footnotemark[2]  \qquad \qquad Peter Richt\'{a}rik \footnotemark[3] }
\begin{document}

\maketitle
\renewcommand{\thefootnote}{\fnsymbol{footnote}}

\footnotetext[2]{School of Mathematics, The Maxwell Institute
for Mathematical Sciences, University of Edinburgh,(e-mail: gowerrobert@gmail.com)}
\footnotetext[3]{School of Mathematics, The Maxwell Institute
for Mathematical Sciences, University of Edinburgh, United Kingdom (e-mail: peter.richtarik@ed.ac.uk)}
\renewcommand{\thefootnote}{\arabic{footnote}}

\begin{abstract}
We develop a novel, fundamental and surprisingly simple {\em randomized iterative method} for solving consistent linear systems.  Our method has six different but equivalent interpretations: sketch-and-project, constrain-and-approximate, random intersect, random linear solve, random update and random fixed point.
By varying its two parameters---a positive definite matrix (defining geometry), and a random matrix (sampled in an i.i.d.\ fashion in each iteration)---we recover a comprehensive array of  well known algorithms as special cases, including the randomized Kaczmarz method, randomized Newton method, randomized coordinate descent method and  random Gaussian pursuit. We naturally also obtain variants of all these methods using blocks and importance sampling. However, our method allows for a much wider selection of these two parameters, which leads to a number of new specific methods. We prove exponential convergence of the {\em expected norm of the error} in a single theorem, from which existing complexity results for known variants can be obtained.  However,  we also give an exact formula for the evolution   of the expected iterates, which allows us to give {\em lower bounds} on the convergence rate. 
\end{abstract}



\pagestyle{myheadings}
\thispagestyle{plain}

\section{Introduction}


The need to solve linear systems of equations is ubiquitous in essentially all quantitative areas of human endeavour, including industry and science. Linear systems are a central problem in numerical linear algebra, and play an important role in computer science, mathematical computing, optimization, signal processing, engineering, numerical analysis, computer vision, machine learning,  and many other fields. 

For instance, in the field of large scale optimization, there is a growing interest in inexact and approximate Newton-type methods for ~\cite{Dembo1982,Eisenstat1994b,Bellavia1998,Zhao2010a,Wang2013,Gondzio2013}, which can benefit from fast subroutines for calculating approximate solutions of linear systems.  In machine learning, applications arise for the problem of finding optimal configurations in Gaussian Markov Random Fields \cite{GMRFbook}, in graph-based semi-supervised learning and other graph-Laplacian problems \cite{Bengio+al-ssl-2006}, least-squares SVMs, Gaussian processes and more. 

In a large scale setting, direct methods are generally not competitive when compared to iterative approaches. While classical iterative methods are deterministic, 
recent breakthroughs suggest that randomization can play a powerful role in the design and analysis of efficient algorithms~\cite{Strohmer2009,Leventhal2010,Needell2010,Drineas2011,Zouzias2013,Lee2013,Ma2015, Richtarik2015} which are in many situations competitive  or better than existing deterministic methods.

\subsection{Contributions}

Given a real matrix $A \in \R^{m \times n}$  and a real vector $b \in \R^m$, in this paper we consider the linear system
\begin{equation}\label{eq:Axb}Ax =b.\end{equation}
We shall assume throughout that the system is consistent:  there exists $x^*$ for which $Ax^*=b$.   

We now comment on the main contribution of this work:

\textbf{1. New method.} We develop a novel, fundamental, and surprisingly simple  {\em randomized iterative method} for solving \eqref{eq:Axb}.

\textbf{2. Six equivalent formulations.} Our method allows for several seemingly different but nevertheless equivalent formulations. First, it can be seen as a {\em sketch-and-project} method, in which the system \eqref{eq:Axb} is replaced by its {\em random sketch}, and then the current iterate is projected onto the solution space of the sketched system. We can also view  it as a {\em constrain-and-approximate} method, where we constrain the next iterate to live in a particular random affine space passing through the current iterate, and then pick the point from this subspace which best approximates the optimal solution. Third,  the method can be seen as an   iterative solution of a sequence of random (and simpler) linear equations. The method also allows for a simple geometrical interpretation: the new iterate is defined as the unique intersection of two random affine spaces which are orthogonal complements. \rob{The fifth viewpoint gives a closed form formula for the  {\em random update} which needs to be applied to the current iterate in order to arrive at the new one. Finally, the method can be seen as a {\em random fixed point iteration.}}

\textbf{3. Special cases.} These multiple viewpoints enrich our interpretation of the method, and enable us to draw previously unknown links between several existing algorithms. Our algorithm has two parameters, an $n\times n$ positive definite matrix $B$ defining geometry of the space, and a random matrix $S$.  Through combinations of these two parameters, in special cases our method recovers several well known algorithms.  For instance, we recover the randomized Kaczmarz method of Strohmer and Vershyinin~\cite{Strohmer2009}, randomized coordinate descent method  of Leventhal and Lewis~\cite{Leventhal2010}, random pursuit~\cite{Nesterov2011a,Stich2014a,Stich2014,S.U.StichC.L.Muller2014} (with exact line search), and the stochastic Newton method recently proposed by Qu et al \cite{Richtarik2015}. However, our method is more general, and leads to i) various generalizations and improvements of the aforementioned methods (e.g., block setup, importance sampling), and ii) completely new methods. Randomness enters our framework in a very general form, which allows us to obtain a {\em Gaussian Kaczmarz method}, {\em Gaussian descent},  and more. 

 \textbf{4. Complexity: general results.} When $A$ has full column rank, our framework allows us to determine the complexity of these methods using a single analysis. Our main results are summarized in Table~\ref{tab:complexity},
where $\{x^k\}$ are the iterates of our method,  $Z$ is a random matrix dependent on the data matrix $A$, parameter matrix $B$ and random parameter matrix $S$, defined as \begin{equation}\label{eq:Z-first}
Z \eqdef A^TS(S^TAB^{-1}A^TS)^{\dagger}S^TA,
\end{equation}
\rob{where $\dagger$ denotes the (Moore-Penrose) pseudoinverse\footnote{\rob{Every (not necessarily square) real matrix $M$ has a real pseudoinverse. In particular, in this paper we will use the following properties of the pseudoinverse: $ M M^\dagger M = M$, $M^\dagger M M^\dagger = M$,  $(M^T M)^\dagger M^T  = M^\dagger$, $(M^T)^\dagger = (M^\dagger)^T$ and $(M M^T)^\dagger = (M^\dagger)^T M^\dagger$.}}. Moreover,} $\norm{x}_B \eqdef  \sqrt{\dotprod{x,x}_B}$, where  $\dotprod{x,y}_B \eqdef x^TBy$, for all $x,y \in \R^{n}$. \rob{It can be deduced from the properties of the pseudoinverse that $Z$ is necessarily symmetric and positive semidefinite\footnote{ Indeed, it suffices to use the identity $(M M^T)^\dagger = (M^\dagger)^T M^\dagger$ with $M=S^T AB^{-1/2}$.}.} 

\rob{As we shall see later, we will often consider setting $B=I$, $B=A$ (if $A$ is positive definite) or $B=A^TA$ (if $A$ is of full column rank). In particular, we first} 
show that the convergence rate  $\rho$ is always bounded between zero and one. We also show that as soon as $\E{Z}$ is invertible (which can only happen if $A$ has full column rank, which then implies that $x^*$ is unique), we have $\rho<1$, and the method converges.
Besides establishing a bound involving the {\em expected norm of the error} (see \rob{the} last line of Table~\ref{tab:complexity}), we also obtain bounds involving the {\em norm of the expected error} (second line of Table~\ref{tab:complexity}). Studying the expected sequence of iterates directly is very fruitful, as it allows us to establish an {\em exact characterization} of the evolution of the expected iterates (see \rob{the} first line of Table~\ref{tab:complexity}) through a {\em linear fixed point iteration}. 


Both of these theorems on the convergence of the error can be recast as iteration complexity bounds.
For instance, using standard arguments, from Theorem~\ref{theo:normEconv} in Table~\ref{tab:complexity} we observe that for a given $\epsilon >0$ we have that
\begin{equation} \label{eq:itercomplex}k \geq \frac{1}{1-\rho} \log\left(\frac{1}{\epsilon}\right) \quad \Rightarrow \quad \norm{\E{x^k-x^*}}_B \leq \epsilon \norm{x^0-x^*}_B.
\end{equation}


\textbf{5. Complexity: special cases.} Besides these generic results, which hold without any major restriction on the sampling matrix $S$ (in particular, it can be either discrete or continuous), we give a specialized result applicable to discrete sampling matrices $S$ (see Theorem~\ref{theo:convsingleS}). In the special cases for which rates are known, our analysis  recovers the existing rates. 


\begin{table}
\centering
\begin{tabular}{|c|c|}
\hline
& \\
 $\E {x^{k+1} -x^{*}} = \left(I  - B^{-1}\E{Z}\right) \E{x^{k} - x^{*}}  $ & Theorem 4.1\\
 & \\
$\norm{\E {x^{k+1} -x^{*}}}_B^2 \leq \rho^2 \; \cdot \; \norm{\E{x^{k} - x^{*}}}_B^2
 $ & Theorem~\ref{theo:normEconv}\\
 & \\
$ \E {\norm{x^{k+1} -x^{*} }_B^2 } \leq \rho \;\cdot\; \E{ \norm{x^{k} - x^{*}}_B^2}$  & Theorem~\ref{theo:Enormconv}\\
& \\
 \hline
 \end{tabular}
 \caption{Our main complexity results. The convergence rate is: $\rho = 1- \lambda_{\min}(B^{-1/2}\E{Z}B^{-1/2}).$}
 \label{tab:complexity}
 \end{table}

\textbf{6. Extensions.} Our approach opens up many avenues for further development and research. For instance, it is possible to extend the results to the case when $A$ is not necessarily of full column rank.  Furthermore, as our results hold for a wide range of distributions, new and efficient variants of the general method can be designed for problems of specific structure by fine-tuning the stochasticity to the structure. Similar ideas can be applied to design randomized iterative algorithms for finding the inverse of a very large matrix.

\subsection{Background and Related Work}


\rob{The literature on solving linear systems via iterative methods is vast and has long history \cite{Kelley95,Saad2003}.  For instance, the Kaczmarz method, in which one cycles through the rows of the system and each iteration is formed by projecting the current point to the hyperplane formed by the active row, dates back to the 30's~\cite{Kaczmarz1937}. The Kaczmarz method is just one example of an array of row-action methods for linear systems (and also, more generally, feasibility and optimization problems)  which were studied in the second half of the 20th century~\cite{rowaction1981}. }

Research into the Kaczmarz method was in 2009 reignited by Strohmer and Vershynin~\cite{Strohmer2009}, who gave a brief and elegant proof that a randomized thereof enjoys an exponential error decay (also know as ``linear convergence''). This has triggered much research into developing and analyzing randomized linear solvers. 

\rob{It should be mentioned at this point that the randomized Kaczmarz (RK) method  arises as a special case (when one considers quadratic objective functions) of the stochastic gradient descent (SGD) method for {\em convex optimization} which can be traced back to the seminal work of  Robbins and Monro's  on stochastic approximation~\cite{RobbinsMonro:1951}. Subsequently, intensive research went into studying various extensions of the SGD method. However, to the best of our knowledge, no complexity results with exponential error decay were established prior to the aforementioned work of  Strohmer and Vershynin~\cite{Strohmer2009}. This is the reason behind our choice of  \cite{Strohmer2009} as the starting point of our discussion.

Motivated by the results of Strohmer and Vershynin~\cite{Strohmer2009}, Leventhal and Lewis~\cite{Leventhal2010} utilize similar techniques to establish the first bounds for {\em randomized coordinate descent} methods for solving systems with positive definite matrices, and systems arising from least squares problems~\cite{Leventhal2010}.  These bounds are similar to those for the RK method. This development was later picked up by the optimization and machine learning communities, and much progress has been made in generalizing these early results in countless ways to various structured  convex optimization problems. For a brief up to date account of the  development in this area, we refer the reader to \cite{APPROX, ALPHA} and the references therein. 
}

The RK method and its analysis have been further extended to  the least-squares problem~\cite{Needell2010,Zouzias2013} and the  block setting~\cite{Needell2012a,Needell2014}. In~\cite{Ma2015} the authors extend the randomized coordinate descent and the RK methods to the problem of solving underdetermined systems.  The authors of~\cite{Ma2015,Ramdas2014} analyze side-by-side the  randomized coordinate descent  and RK method,  for least-squares, using a convenient notation in order to point out their similarities. Our work takes the next step, by analyzing these, and many other methods, through a genuinely single analysis. Also in the spirit of unifying the analysis of different methods, in~\cite{Oswald2015} the authors provide a unified analysis of iterative Schwarz methods and Kaczmarz methods. 

The use of random Gaussian directions as search directions  in zero-order (derivative-free) minimization algorithm was recently suggested~\cite{Nesterov2011a}. More recently, Gaussian directions have been combined with exact and inexact line-search into a single {\em random pursuit} framework~\cite{S.U.StichC.L.Muller2014}, and further utilized within a randomized variable metric method~\cite{Stich2014,Stich2014a}.
 

\section{One Algorithm in Six Disguises} Our method has {\em two parameters}: i) an $n\times n$ positive definite matrix $B$ which is used to define the $B$-inner product and the induced $B$-norm by \begin{equation} \label{eq:B-innerprod}\langle x, y\rangle_{B}\eqdef \langle Bx,y\rangle, \qquad \|x\|_B \eqdef \sqrt{ \langle x, x \rangle_B},\end{equation}
where $\langle \cdot,\cdot \rangle$ is the standard Euclidean inner product,  and ii) a random matrix $S\in \R^{m\times q}$, to be drawn in an i.i.d.\ fashion at each iteration. We stress that we do not restrict the number of columns of $S$; indeed, we even allow $q$ to vary (and hence,  $q$ is a random variable).

\subsection{Six Viewpoints}

Starting from $x^k\in\R^n$, our method draws a random matrix $S$ and uses it to generate a new point $x^{k+1}\in\R^n$. This iteration can be formulated in {\em six seemingly different but equivalent ways:}

\paragraph{1. Sketching Viewpoint: Sketch-and-Project} $x^{k+1}$ is the nearest point to $x^k$ which solves a {\em sketched} version of the original linear system:
\begin{equation}
\boxed{\quad x^{k+1} \quad = \quad \arg\min_{x\in \R^n} \norm{x- x^{k}}_B^2 \quad \mbox{subject to} \quad  S^T Ax = S^T b \quad} \label{eq:NF} 
\end{equation}
This viewpoint arises very naturally. Indeed, since the original system \eqref{eq:Axb} is assumed to be complicated, we replace it by a simpler system---a {\em random sketch} of the original system \eqref{eq:Axb}---whose solution set $\{x \;|\; S^T Ax = S^T b\}$ contains all solutions of the original system. However, this system will typically have many solutions, so in order to define a method, we need a way to select one of them. The idea is to try to preserve as much of the information learned so far as possible, as condensed in the current point $x^k$. Hence, we  pick the solution which is closest to $x^k$. 

\paragraph{2. Optimization Viewpoint: Constrain-and-Approximate} $x^{k+1}$ is the best approximation of $x^*$ in a random space passing through $x^k$:
\begin{equation}\boxed{\; x^{k+1} \; = \; \arg\min_{x\in \R^n} \norm{x\phantom{^k}- x^{*}}_B^2 \quad \mbox{subject to} \quad x = x^{k} + B^{-1}A^TS y, \quad y \;\text{is free} \;} \label{eq:RF}
\end{equation}
The above step has the following interpretation\footnote{Formulation~\eqref{eq:RF} is similar to the framework often used to describe Krylov methods~\cite[Chapter 1]{Liesen2014}, which is
\[x^{k+1} \eqdef \arg\min_{x\in \R^n} \norm{x- x^{*}}_B^2 \quad \mbox{subject to} \quad x \in x^{0} + \K_{k+1},\]
where $\K_{k+1}\subset \R^n$ is a $(k+1)$--dimensional subspace.
Note that the constraint $x \in x^{0}+\K_{k+1}$ is an affine space that contains $x^{0}$, as opposed to $x^{k}$ in our formulation~\eqref{eq:RF}.
The objective $\|x-x^{*}\|^2_B $ is a generalization of the residual, where $B=A^TA$ is used to characterize minimal residual methods~\cite{Paige1975,Saad1986} and $B=A$ is used to describe the Conjugate Gradients method~\cite{Hestenes1952}. Progress from iteration to the next is guaranteed by using expanding nested search spaces at each iteration, that is, $\K_k \subset \K_{k+1}.$ In our setting, progress is enforced by using $x^{k}$ as the displacement term instead of $x^{0}.$  This also allows for a simple recurrence for updating $x^{k}$ to arrive at $x^{k+1}$, which facilitates the analyses of the method. In the Krylov setting, to arrive at an explicit recurrence, one needs to carefully select a basis for the nested spaces that allows for short recurrence. }. We  choose a random affine space containing $x^k$, and constrain our method to choose the next iterate from this space. We then do as well as we can on this space; that is, we pick $x^{k+1}$ as the point which best approximates $x^*$. Note that $x^{k+1}$ does not depend on which solution $x^*$ is used in \eqref{eq:RF} (this can be best seen by considering the geometric viewpoint, discussed next).

 \begin{figure}[!h] \centering
\includegraphics[width =7cm]{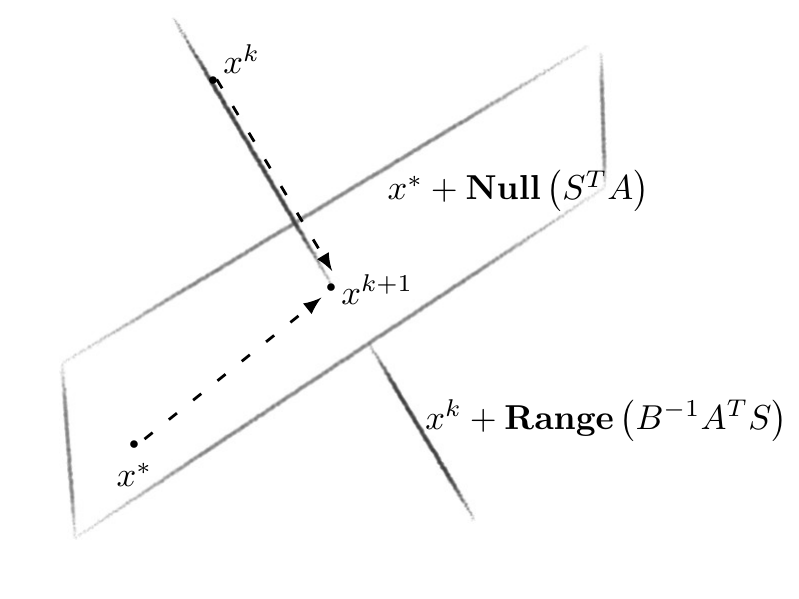}
\caption{\footnotesize The geometry of our algorithm. The next iterate, $x^{k+1}$, arises as the intersection of two random affine spaces: $x^k + \myRange{B^{-1}A^T S}$ and $x^* + \Null{S^TA}$ (see \eqref{eq:geometry}). The spaces are orthogonal complements of each other with respect to the $B$-inner product, and hence $x^{k+1}$ can equivalently be written as the projection, in the $B$-norm, of $x^k$ onto $x^* +\Null{S^T A}$ (see \eqref{eq:NF}), or the projection of $x^*$  onto $x^k +\myRange{B^{-1}A^T S}$ (see \eqref{eq:RF}). The intersection $x^{k+1}$ can also be expressed as the solution of a system of linear equations (see \eqref{eq:2systems}). Finally, the new error $x^{k+1}-x^*$ is the projection, with respect to the $B$-inner product, of the current error $x^k-x^*$ onto $\Null{S^TA}$. This gives rise to a random fixed point formulation (see \eqref{eq:xZupdate}).}
\label{fig:proj}
\end{figure}

\paragraph{3. Geometric viewpoint: Random Intersect} $x^{k+1}$ is the (unique) intersection of two affine spaces:
\begin{equation} \label{eq:geometry}
\boxed{ \quad \{x^{k+1}\} \quad =\quad  \left( x^* + \Null{S^T A}\right) \quad \bigcap \quad \left(x^k + \myRange{B^{-1}A^T S}\right) \quad}
\end{equation}
First, note that the first affine space above does not depend on the choice of $x^*$ from the set of optimal solutions of \eqref{eq:Axb}. A basic result of linear algebra says that the nullspace of an arbitrary matrix is the orthogonal complement of the range space of its transpose. Hence, whenever we have $h\in \Null{S^TA}$ and $y\in \R^q$, where $q$ is the number of rows of $S$, then $\langle h, A^T S y \rangle = 0$. It follows that the two spaces in \eqref{eq:geometry} are orthogonal complements with respect to the $B$-inner product  and as such, they intersect at a unique point (see Figure~\ref{fig:proj}).

 \paragraph{4. Algebraic viewpoint: Random Linear Solve} Note that $x^{k+1}$ is the (unique) solution (in $x$) of a linear system (with variables $x$ and $y$):
\begin{equation}\label{eq:2systems}
\boxed{\quad x^{k+1} \quad = \quad \text{solution of }\quad  S^T A x = S^T b, \quad x = x^k + B^{-1}A^T S y\quad } \end{equation}
This system is clearly equivalent to \eqref{eq:geometry}, and can alternatively be written as:
\begin{equation}\label{eq:view-algebraic}\begin{pmatrix}S^T A & 0  \\
B & -A^T S\end{pmatrix} \begin{pmatrix}x \\y\end{pmatrix} = \begin{pmatrix}S^T b \\ B x^k\end{pmatrix}.\end{equation}
Hence, our method reduces the solution of the (complicated) linear  system \eqref{eq:Axb} into a sequence of (hopefully simpler) random systems of the form \eqref{eq:view-algebraic}. 

\rob{
 \paragraph{5. Algebraic viewpoint: Random Update} 
 By plugging the second equation in \eqref{eq:2systems} into the first, we eliminate $x$ and obtain the system $(S^T A B^{-1}A^T S) y = S^T(b-Ax^k)$. Note that for all solutions $y$ of this system we must have $x^{k+1} = x^k + B^{-1}A^T S y$. In particular,  we can choose the solution $y=y^k$ of minimal Euclidean norm, which is given by $y^k = (S^T A B^{-1}A^T S)^{\dagger}S^T(b-Ax^k)$, where $^{\dagger}$ denotes the Moore-Penrose pseudoinverse. This leads to an expression for $x^{k+1}$ with an explicit form of the {\em random update} which must be applied to $x^k$ in order to obtain $x^{k+1}$:
\begin{equation}\label{eq:MP} 
\boxed{\quad x^{k+1} = x^k - B^{-1}A^T S(S^T A B^{-1}A^T S)^{\dagger}S^T(Ax^k-b) \quad }\end{equation}
In some sense, this form is the standard: it is customary for iterative techniques to be written in the form $x^{k+1}=x^k + d^k$, which is precisely what \eqref{eq:MP}  does.
}

\paragraph{6. Analytic viewpoint: Random Fixed Point} 
Note that  iteration~\eqref{eq:MP} can be written as 
\begin{equation}\label{eq:xZupdate}
\boxed{\quad x^{k+1} -  x^* \quad =\quad (I-  B^{-1}Z)(x^{k} -x^*)\quad}
\end{equation} 
where $Z$ is defined in \eqref{eq:Z-first} and where we used the fact that $Ax^{*} =b$. Matrix $Z$ plays a central role in our analysis, and can be used to construct explicit projection matrices of the two projections depicted in Figure~\ref{fig:proj}.

The equivalence between these \rob{six} viewpoints is formally captured in the next statement.

\begin{theorem}[\rob{Equivalence}]\label{thm:equivalence} The \rob{six} viewpoints are equivalent: they all produce the same (unique) point $x^{k+1}$.
\end{theorem}
\begin{proof} The proof is simple, and follows  directly from the above discussion. In particular, see the caption of Figure~\ref{fig:proj}.
\end{proof}

\subsection{Projection Matrices}

\rob{
 
In this section we state a few key properties of matrix $Z$. This will shed light on the previous discussion and will also be useful later in the convergence proofs. 

Recall that $S$ is a $m\times q$ random matrix (with $q$ possibly being random),  and that $A$ is an $m\times n$ matrix. Let us define the random quantity \begin{equation}\dd \eqdef \Rank{S^T A}\end{equation} and notice that  $d\leq \min\{q,n\}$,
\begin{equation} \label{eq:dimproj}
\dim \left(\myRange{B^{-1}A^T S}\right) = \dd, \qquad \text{and}\qquad \dim\left(\Null{S^TA}\right) = n-\dd.
\end{equation}

\begin{lemma}\label{lem:BZ} With respect to the geometry induced by the $B$-inner product, we have that
\begin{enumerate}
\item[(i)] $B^{-1}Z$ projects orthogonally  onto the $\dd$--dimensional subspace $\myRange{B^{-1}A^TS}$
\item[(ii)] $(I-B^{-1}Z)$ projects orthogonally onto $(n-\dd)$--dimensional subspace $\Null{S^TA}.$  
\end{enumerate}
\end{lemma}
\begin{proof}
\rob{For any matrix $M$, the pseudoinverse satisfies the identity $M^{\dagger}M M^{\dagger} =M^{\dagger}$. Using this with $M=S^TAB^{-1}A^TS$,  we get
 \begin{align}\label{eq:projectionxx}(B^{-1}Z)^2 &\overset{\eqref{eq:Z-first}}{=}B^{-1}A^TS(S^TAB^{-1}A^TS)^{\dagger}S^TA B^{-1} A^TS(S^TAB^{-1}A^TS)^{\dagger}S^TA \nonumber\\
 & = B^{-1}A^TS(S^TAB^{-1}A^TS)^{\dagger}S^TA \overset{\eqref{eq:Z-first}}{=} B^{-1}Z,
 \end{align} and thus both $B^{-1}Z$ and $I-B^{-1}Z$ are projection matrices.} In order to establish that $B^{-1}Z$ is an orthogonal projection with respect to the $B$-inner product (from which it follows that $I-B^{-1}Z$ is), we will show that
\[B^{-1}Z (B^{-1}A^TS) =B^{-1}A^TS, \quad \quad \mbox{and}\quad \quad B^{-1}Z y =0, \; \forall y \in \Null{S^TA}.\]
The second relation is trivially satisfied. In order to establish the first relation, it is enough to use two further properties of the pseudoinverse:   $(M^T M)^\dagger M^T  = M^\dagger$ and $MM^\dagger M = M$, both with $M=B^{-1/2}A^T S$. Indeed,
\begin{eqnarray*} B^{-1}Z (B^{-1}A^TS) &\overset{\eqref{eq:Z-first}}{=} & B^{-1/2} M (M^T M)^\dagger M^T M \\
&=& B^{-1/2}MM^\dagger M\\
& = & B^{-1/2}M = B^{-1}A^T S.\end{eqnarray*}
\end{proof}
}

\rob{This lemma sheds additional light on Figure~\ref{fig:proj} as it gives explicit expressions for the associated projection matrices.}  The result also implies that $I-B^{-1}Z$ is a {\em contraction} with respect to the $B$-norm, \rob{which means that the random fixed point iteration~\eqref{eq:xZupdate} has only very little room not to work.} While $I-B^{-1}Z$ is not a strict contraction, under some reasonably weak assumptions on $S$ it will be a strict contraction in expectation, which  ensures convergence. \rob{We shall state these assumptions and develop the associated convergence theory for our method in Section~\ref{sec:convergence} and Section~\ref{sec:discrete}.}

\section{Special Cases: Examples} \label{sec:examples}


In this section we briefly mention how by selecting the parameters $S$ and $B$ of our method we recover several existing methods. The list is by no means comprehensive and merely serves the purpose of an illustration of the flexibility of our algorithm. All the associated complexity results we present in this section, can be recovered from Theorem~\ref{theo:convsingleS}, presented later in Section~\ref{sec:discrete}. 

\subsection{The One Step Method} 

When $S$ is an $m\times m$ invertible matrix with probability one, then the system $S^TAx=S^Tb$ is equivalent to solving $Ax=b,$ thus the  solution to~\eqref{eq:NF} must be $x^{k+1}=x^*$, independently of matrix $B.$
Our convergence theorems also predict this one step behaviour, since $\rho =0$ (see Table~\ref{tab:complexity}).


\subsection{Random Vector Sketch}
\rob{When $S = s \in \R^{m}$ is restricted to being a random column vector, then from~\eqref{eq:MP} a step of our method is given by
\begin{equation}\label{eq:vecsketch}x^{k+1} = x^{k}  - \frac{s^T(A x^{k}-b)}{ s^TAB^{-1}A^Ts} B^{-1}A^Ts,
\end{equation}
if $A^Ts \neq 0$ and $x^{k+1} =x^k$ otherwise. This is because the pseudo inverse of a scalar $\alpha \in \R$ is given by \[\alpha^{\dagger} = 
\begin{cases}
1/\alpha  & \mbox{if } \alpha \neq 0\\
0 &  \mbox{if } \alpha = 0.
\end{cases}
\]
Next we describe several well known specializations of the random vector sketch and for brevity, we write the updates in the form of~\eqref{eq:vecsketch} and leave implicit that when the denominator is zero, no step is taken.}
\subsection{Randomized Kaczmarz}

If we choose $S=e^i$ (unit coordinate vector in $\R^m$) and $B=I$ (the identity matrix), in view of \eqref{eq:NF} we obtain the method:
\begin{equation} \label{eq:RKintro} x^{k+1} = \arg\min_{x\in \R^n} \norm{x- x^{k}}_2^2 \quad \mbox{ subject to } \quad  A_{i:}x =b_{i}.
\end{equation}
Using~\eqref{eq:MP}, these iterations can be calculated with
\begin{equation}\label{eq:RKiterate} \boxed{x^{k+1} = x^{k} - \frac{A_{i:} x^{k}-b_{i}}{\norm{A_{i:}}_2^2}(A_{i:})^T} \end{equation}

\paragraph{Complexity} When $i$ is selected at random, this is the randomized Kaczmarz (\emph{RK}) method~\cite{Strohmer2009}. A specific non-uniform probability distribution for $S$ can yield simple and easily interpretable (but not necessarily optimal) complexity bound. In particular, by selecting $i$ with probability proportional to the magnitude of row $i$ of $A$, that is $p_i = \norm{ A_{i:} }_2^2/\norm{A}_F^2$,  it follows from Theorem~\ref{theo:convsingleS} that RK enjoys the following complexity bound:
\begin{equation}\label{eq:RKconv}\E {\norm{x^{k} -x^{*} }_2^2 } \leq  \left(1  -  \frac{\lambda_{\min}\left(A^T A \right)}{\norm{A}_F^2} \right)^k \norm{x^{0} - x^{*}}_2^2.\end{equation}
This result was first established by Strohmer and Vershynin \cite{Strohmer2009}. 
We also provide new convergence results in Theorem~\ref{theo:normEconv}, based on the convergence of the norm of the expected error. Theorem~\ref{theo:normEconv} applied to the RK method gives
\begin{equation}\label{eq:RKconv2}\norm{\E {x^{k} -x^{*} } }_2^2 \leq  \left(1  -  \frac{\lambda_{\min}\left(A^T A \right)}{\norm{A}_F^2} \right)^{2k} \norm{x^{0} - x^{*}}_2^2.\end{equation}
Now the convergence rate appears squared, which is a better rate, though, the expectation has moved inside the norm, which is a weaker form of convergence. 
  
Analogous results for the convergence of the norm of the expected error holds  for all the methods we present, though we only illustrate this with the RK method.

\paragraph{Re-interpretation as SGD with exact line search} Using the Constrain-and-Approximate formulation~\eqref{eq:RF}, the randomized Kaczmarz method can also be written as
\[x^{k+1} =\arg \min_{x\in \R^n} \norm{x- x^*}_2^2 \quad \mbox{ subject to } \quad  x = x^k + y(A_{i:})^T, \quad y \in \R, \]
with probability $p_i$. Writing the least squares function $f(x) = \tfrac{1}{2}\|Ax-b\|_2^2$ as
\[f(x) = \sum_{i=1}^m p_i f_i(x), \qquad f_i(x) = \frac{1}{2 p_i}(A_{i:}x - b_i)^2,\]
we see that the random vector $\nabla f_i(x) = \tfrac{1}{p_i}(A_{i:}x-b_i) (A_{i:})^T$ is an unbiased estimator of the  gradient of $f$ at $x$. That is,  $\E{\nabla f_i(x)} = \nabla f(x)$. Notice that RK takes a step in the direction $-\nabla f_i(x)$. This is true even when $A_{i:}x-b_i = 0$, in which case, the RK does not take any step. Hence, RK takes a step in the direction of the negative stochastic gradient. This means that it is equivalent to the Stochastic Gradient Descent (SGD) method. However,  the stepsize choice is very special: RK chooses the stepsize which leads to the point which is closest to $x^*$ in the Euclidean norm.



\subsection{Randomized Coordinate Descent: positive definite case}  \label{sec:shs98ss}
 
 If $A$ is positive definite, then we can choose $B= A$ and $S = e^i$  in~\eqref{eq:NF}, which results in 
\begin{equation} \label{eq:CDpdintro}
x^{k+1} \eqdef \arg\min_{x\in \R^n} \norm{x- x^{k}}_A^2 \quad \mbox{subject to} \quad (A_{i:})^T x =b_{i},
\end{equation}
where we used the symmetry of $A$ to get $(e^i)^TA = A_{i:}=(A_{:i})^T.$
The solution to the above, given by~\eqref{eq:MP}, is 
\begin{equation}\label{eq:09j0s9jsss}
\boxed{x^{k+1} = x^{k} -  \frac{(A_{i:})^Tx^{k}-b_i}{A_{ii}}e^{i}}\end{equation}
\paragraph{Complexity} When $i$ is chosen randomly, this is the  \emph{Randomized CD} method (CD-pd). Applying Theorem~\ref{theo:convsingleS}, we see the probability distribution $p_i = A_{ii}/\Tr{A}$
results in a convergence with
\begin{equation}\label{eq:CDposconv}
\E {\norm{x^{k} -x^{*} }_{A}^2 } \leq \left(1-  \frac{\lambda_{\min}\left(A \right)}{\Tr{A}}\right)^k \norm{x^{0} - x^{*}}_{A}^2.\end{equation}
  This result was first established by Leventhal and Lewis~\cite{Leventhal2010}.
  
\paragraph{Interpretation} Using the Constrain-and-Approximate formulation~\eqref{eq:RF}, this method can be interpreted as
\begin{equation} \label{eq:CDpdRF}
x^{k+1} =\arg \min \|x-x^*\|_A^2\quad \mbox{ subject to } \quad  x = x^k + y e^i, \quad y \in \R, \end{equation}
with probability $p_i$. It is easy to check that the function $f(x)=\tfrac{1}{2}x^T A x \rob{-} b^T x$ satisfies: $\|x-x^*\|_A^2 = 2f(x) + b^T x^*$. Therefore,  \eqref{eq:CDpdRF} is equivalent to
\begin{equation} \label{eq:CDpdRFxx}
x^{k+1} =\arg \min f(x) \quad \mbox{ subject to } \quad  x = x^k + y e^i, \quad y\in \R. \end{equation}
The iterates~\eqref{eq:09j0s9jsss} can also be written as
\[x^{k+1} = x^k - \frac{1}{L_i}\nabla_{i} f(x^k) e^i,\]
where $L_i = A_{ii}$ is the Lipschitz constant of the gradient of $f$ corresponding to coordinate $i$ and $\nabla_{i} f(x^k)$ is the $i$th partial derivative of $f$ at $x^k$.

\subsection{Randomized Block Kaczmarz}

Our framework also extends to new block formulations of the randomized Kaczmarz method. Let   $R$ be a random subset of $[m]$ and let $S =I_{:R}$ be a column concatenation of the columns of the $m\times m$ identity matrix $I$ indexed by $R$. Further, let $B=I$. Then \eqref{eq:NF} specializes to

\[
 x^{k+1}  =  \arg\min_{x\in \R^n} \norm{x- x^{k}}_2^2 \quad \mbox{subject to} \quad  A_{R:}x =  b_{R} .
\]
In view of~\eqref{eq:MP}, this can be equivalently written as 
\begin{equation}\label{eq:blockRK}\boxed{x^{k+1}=x^k - (A_{R:})^T (A_{R:} (A_{R:})^T)^{\dagger}(A_{R:}x^k - b_{R})} \end{equation}

\paragraph{Complexity} From Theorem~\ref{theo:Enormconv} we obtain the following new complexity result:

\[\E{\|x^{k}-x^*\|_2^2} \leq  \left(1-\lambda_{\min}\left(\E{(A_{R:})^T (A_{R:} (A_{R:})^T)^{\rob{\dagger}}A_{R:}}\right)\right)^k\|x^0-x^*\|^2_2.\]

\rob{
To obtain a more meaningful convergence rate, we would need to bound the smallest eigenvalue of $\E{(A_{R:})^T (A_{R:} (A_{R:})^T)^{\rob{\dagger}}A_{R:}}.$ This has been done in~\cite{Needell2012a,Needell2014} when the image of $R$ defines a row paving of $A$.  Our framework paves the way for analysing the convergence of new block methods for a large set of possible random subsets $R,$ including, for example, overlapping partitions. }

 \subsection{Randomized Newton: positive definite case}
  
 If $A$ is symmetric positive definite, then we can choose $B= A$ and $S = I_{:C}$, a column concatenation of the columns of $I$ indexed by $C$, which is a random subset  of $[n]$.  In view of~\eqref{eq:NF}, this results in 
\begin{equation} \label{eq:CDpdblock}
x^{k+1} \eqdef \arg\min_{x\in \R^n} \norm{x- x^{k}}_A^2 \quad \mbox{subject to} \quad (A_{:C})^T x =b_{C}.
\end{equation}
In view of \eqref{eq:MP}, we can equivalently write the method as
  \begin{equation} \boxed{ \quad x^{k+1}  \quad = \quad x^k -  I_{:C} ((I_{:C})^T A I_{:C})^{-1}  (I_{:C})^T (Ax^k - b) \quad }\label{eq:Method1}
 \end{equation}

\paragraph{Complexity} Clearly, iteration \eqref{eq:Method1} is well defined  as long as $C$ is nonempty with probability 1. Such $C$  is in \cite{Richtarik2015} referred to by the name ``non-vacuous'' sampling. From Theorem \ref{theo:Enormconv} we obtain the following convergence rate:
\begin{eqnarray}\label{eq:Method1xxx}
\E { \norm{x^{k} -x^{*}}_{A}^2 } & \leq & \rho^k \|x^0 - x^*\|_A^2 \notag \\
&=&
\left(1-  \lambda_{\min}\left( \E{ I_{:C}  ((I_{:C})^T A I_{:C})^{-1} (I_{:C})^T A} \right)\right)^k \norm{x^{0} - x^{*}}_{A}^2.
\end{eqnarray}

The convergence rate of this particular method was first established \rob{and studied} in \cite{Richtarik2015}.  Moreover, it was shown in \cite{Richtarik2015} that $\rho<1$ if one additionally assumes that the probability that $i \in C$ is positive for each column $i\in [n]$, i.e., that $C$ is a ``proper'' sampling.


\paragraph{Interpretation}   Using  formulation~\eqref{eq:RF}, and in view of the equivalence between $f(x)$ and $\|x-x^*\|_A^2$ discussed in Section~\ref{sec:shs98ss},  the Randomized Newton method can be equivalently written as
\[x^{k+1} =\arg \min_{x\in \R^n} f(x) \quad \mbox{ subject to } \quad  x = x^k + I_{:C}\, y, \quad y \in \R^{|C|}. \]
The next iterate is determined by advancing from the previous iterate over a subset of coordinates such that $f$ is minimized. Hence, an exact line search is performed in a random $|C|$ dimensional subspace.

Method~\eqref{eq:Method1} was fist studied by Qu et al~\cite{Richtarik2015}, and referred therein as  ``Method~1'', or {\em Randomized Newton Method}. The name comes from the observation that the method inverts random principal submatrices of $A$ and that in the special case when $C=[n]$ with probability 1, it specializes to the Newton method (which in this case converges in a single step). 
\rob{The expression $\rho$ defining the convergence rate of this method is rather involved and it is not immediately obvious what is gained by performing a search in a higher dimensional subspace ($C>1$) rather than in the one-dimensional subspaces ($C=1$), as is standard in the optimization literature. Let us write  $\rho = 1-\sigma_\tau$ in the case when the $C$ is chosen to be a subset of $[n]$ of size $\tau$, uniformly at random. In view of \eqref{eq:itercomplex}, the method takes $\tilde{O}(1/\sigma_\tau)$ iterations to converge, where the tilde notation suppresses logarithmic terms. It was shown in \cite{Richtarik2015}  that $1/\sigma_\tau \leq 1/(\tau \sigma_1)$. That is, one can expect to obtain at least {\em superlinear speedup} in $\tau$ --- this is what is gained by moving to blocks / higher dimensional subspaces. For further details and additional properties of the method we refer the reader to \cite{Richtarik2015}. }

\subsection{Randomized Coordinate Descent: least-squares version} 

By choosing $S=Ae^{i} =:A_{:i}$ as the $i$th column of $A$ and $B=A^TA$, the resulting iterates~\eqref{eq:RF} are given by
\begin{equation} \label{eq:CDLSintro}
x^{k+1} = \arg\min_{x\in \R^n} \norm{Ax-b}_2^2 \quad \mbox{ subject to } \quad  x = x^{k} + y \, e^{i}, \quad y \in \R.
\end{equation}
When $i$ is selected at random, this is the Randomized Coordinate Descent method (\emph{CD-LS})  applied to the least-squares problem: $\min_x \|Ax-b\|_2^2$.  Using~\eqref{eq:MP}, these iterations can be calculated with
\begin{equation}\label{eq:098sh98hs} \boxed{x^{k+1} = x^{k} - \frac{(A_{:i})^T(A x^{k} -b)}{\norm{A_{:i}}_2^2} e^{i} } \end{equation}

\paragraph{Complexity}
Applying Theorem~\ref{theo:convsingleS}, we see that by selecting $i$ with probability proportional to magnitude of column $i$ of $A$, that is $p_i = \norm{ A_{:i} }_2^2/\norm{A}_F^2$, 
results in a convergence with
\begin{equation}\label{eq:CDLSconv}
\E {\norm{x^{k} -x^{*} }_{A^TA}^2 } \leq \rho^k \|x^0-x^*\|^2_{A^T A} =   \left(1  -  \frac{\lambda_{\min}\left(A^T A \right)}{\norm{A}_F^2} \right)^k \norm{x^{0} - x^{*}}_{A^TA}^2.\end{equation}
  This result was first established by Leventhal and Lewis~\cite{Leventhal2010}.
  
\paragraph{Interpretation} 
  Using the Constrain-and-Approximate formulation~\eqref{eq:RF}, the CD-LS method can be interpreted as
\begin{equation} \label{eq:CDLSRF}
x^{k+1} =\arg \min_{x\in \R^n} \norm{x- x^*}_{A^TA}^2 \quad \mbox{ subject to } \quad  x = x^k + y e^i , \quad y \in \R.\end{equation}
The  CD-LS method selects a coordinate to advance from the previous iterate $x^k$, then performs an exact minimization of the least squares function over this line.
This is equivalent to applying coordinate descent to the least squares problem $\min_{x\in \R^n} f(x) \eqdef \tfrac{1}{2}\|Ax-b\|_2^2.$ The iterates~\eqref{eq:CDLSintro} can be written as
\[x^{k+1} =x^k -\frac{1}{L_i}\nabla_i f(x^k) e^i,\] 
 where $L_i \eqdef \norm{A_{:i}}_2^2$ is the Lipschitz constant of the gradient 
 corresponding to coordinate $i$ and $\nabla_{i} f(x^k)$ is the $i$th partial derivative of $f$ at $x^k$.

%

\section{Convergence: General Theory} \label{sec:convergence}

We shall present two complexity theorems: we first study the convergence of $\norm{\E{x^{k}-x^*}}$ , and then move on to analysing the convergence of $\E{\norm{x^{k}-x^*}}$. 


\subsection{Two types of convergence}
The following lemma explains the relationship between the convergence of the norm of the expected error and the expected norm of the error.
\begin{lemma} \label{lem:conv} Let $x\in \R^n$ be a random vector, $\norm{\cdot}$ a  norm induced by an inner product and fix $x^*\in \R^n$. Then
\[ \big\|\E{x-x^*}\big\|^2   = \E{\left\| x - x^* \right\|^2}  - \E{\left\| x - \E{x}\right\|^2}.\]
\end{lemma}
\begin{proof} Note that
$\E{\norm{ x - \E{x}}^2}= \E{\norm{x}^2} - \norm{\E{x}}^2$.
Adding and subtracting $\norm{x^*}^2-2\dotprod{\E{x},x^*}$ from the right hand side and grouping the appropriate terms yields the desired result. 
\end{proof} 
   
To interpret this lemma, note that
$\E{\left\| x - \E{x}\right\|^2} =\sum_{i=1}^n\E{ (x_i - \E{x_i})^2}  = \sum_{i=1}^n \mathbf{Var}(x_i)$,
where $x_i$ denotes the $i$th element of $x.$
This lemma shows that the convergence of $x$ to $x^*$ under the expected norm of the error is a stronger form of convergence than the convergence of the norm of the expected error, as the former also guarantees that the variance of $x_i$  converges to zero, for $i=1,\ldots ,n.$

\subsection{The Rate of Convergence}

All of our convergence theorems (see Table~\ref{tab:complexity}) depend on the  convergence rate
\begin{equation}\label{eq:rho}\rho \eqdef 1 - \lambda_{\min}(B^{-1}\E{Z}) = 1 - \lambda_{\min}(B^{-1/2}\E{Z}B^{-1/2}).\end{equation} 
To show that the rate is meaningful, in Lemma~\ref{lem:rho1} we prove that $0\leq \rho \leq 1$. We also provide a meaningful lower bound for $\rho$.

\begin{lemma}\label{lem:rho1} The quantity $\rho$ defined in~\eqref{eq:rho} satisfies:
\begin{equation} \label{eq:rho_lower}0 \leq 1-\dfrac{\E{d}}{n}\leq \rho \leq 1,\end{equation}
\rob{where $\dd=\Rank{S^TA}$. }
\end{lemma}
\begin{proof}
Since the mapping $A \mapsto \lambda_{\max}(A)$ is convex on the set of symmetric matrices,  by Jensen's inequality we get
\begin{equation}\label{eq:a89h93a8}\lambda_{\max} (\E{B^{-1}Z})= \rob{\lambda_{\max} (B^{-1/2} \E{Z}B^{-1/2})} \leq\E{\lambda_{\max}(B^{-1/2}ZB^{-1/2})}.
\end{equation}
Recalling from Lemma~\ref{lem:BZ} that $B^{-1}Z$ is a projection, \rob{we get \[B^{-1/2}ZB^{-1/2} (B^{-1/2}ZB^{-1/2}) =B^{-1/2}ZB^{-1/2},  \]
whence the spectrum of $B^{-1/2}ZB^{-1/2}$ is contained in $\{0,1\}$. Thus,  $\lambda_{\max}(B^{-1/2}ZB^{-1/2})~\leq~1$,} and from~\eqref{eq:a89h93a8} we conclude that $\lambda_{\max}(B^{-1} \E{Z}) \leq 1$.
The inequality $\lambda_{\min}(B^{-1}\E{Z})~\geq~0$ can be shown analogously using convexity of the mapping $A\mapsto -\lambda_{\min}(A)$. Thus \[\lambda_{\min}(B^{-1}\E{Z}) = \lambda_{\min}(B^{-1/2}\E{Z}B^{-1/2}) \in [0, 1]\] and consequentially $0 \leq \rho \leq 1.$
 As the trace of a matrix is equal to the sum of its eigenvalues, we have
\begin{equation}\label{eq:traceb1} \E{\Tr{B^{-1}Z}} =\Tr{\E{B^{-1}Z}} \geq n \, \lambda_{\min}(\E{B^{-1}Z}). \end{equation}
As $B^{-1}Z$ projects onto a $\dd$--dimensional subspace (Lemma~\ref{lem:BZ}) we have $\Tr{B^{-1}Z} = \dd.$ 
Thus  rewriting~\eqref{eq:traceb1} gives $1-\E{\dd}/n \leq \rho.$
\end{proof}

The lower bound on $\rho$ in item 1 has a natural interpretation which makes intuitive sense. We shall present it from the perspective of the Constrain-and-Approximate formulation~\eqref{eq:RF}. As the dimension ($\dd$) of the search space $B^{-1}A^TS$ increases (see \eqref{eq:dimproj}), the lower bound on $\rho$ decreases, and a faster convergence is possible. For instance, when $S$ is restricted to being a random column vector, as it is in the RK~\eqref{eq:RKiterate}, CD-LS~\eqref{eq:098sh98hs} and CD-pd~\eqref{eq:CDposconv} methods, the convergence rate is bounded with $1 -1/n\leq \rho.$ Using~\eqref{eq:itercomplex}, this translates into the simple iteration complexity bound of $k \geq n \log(1/\epsilon)$. On the other extreme, when the search space is large, then the lower bound  is close to zero, allowing room for the method to be faster.

We now characterize circumstances under which $\rho$ is strictly smaller than one.

\begin{lemma}\label{lem:rho} If $\E{Z}$ is invertible, then $\rho <1$, $A$ has full column rank and $x^*$ is unique.
\end{lemma}
\begin{proof}
Assume that $\E{Z}$ is invertible. First, this means that $B^{-1/2}\E{Z}B^{-1/2}$ is positive definite, which in view of \eqref{eq:rho} means that $\rho <1.$ If $A$ did not have full column rank, then there would be $0\neq x\in \R^n$ such that $Ax=0$. However, we then have $Zx = 0$ and also $\E{Z}x=0$, contradicting the assumption that $\E{Z}$ is invertible. Finally, since $A$ has full column rank, $x^*$ must be unique (recall that we assume throughout the paper that the system $Ax=b$ is consistent).
\end{proof}

\subsection{Exact Characterization and Norm of  Expectation}

We now state a theorem which exactly characterizes the evolution of the expected iterates through a linear fixed point iteration. As a consequence, we obtain a convergence result for the norm of the expected error.  While we do not highlight this in the text, this theorem can be applied to all the particular instances of our general method we detail throughout this paper.


For any $M\in \R^{n\times n}$ let us define
\begin{equation}\label{eq:specnorm}
\|M\|_B \eqdef \max_{\norm{x}_B=1} \|Mx\|_B.
\end{equation}

\begin{theorem}[Norm of expectation]\label{theo:normEconv} For every  $x^*\in \R^n$ satisfying $Ax=b$ we have
\begin{equation} \label{eq:Eerror}
\E {x^{k+1} -x^{*}} = \left(I  - B^{-1}\E{Z}\right) \E{x^{k} - x^{*}}.
 \end{equation}
Moreover, the spectral radius and the induced $B$-norm of the iteration matrix $I-B^{-1}\E{Z}$ are both equal to $\rho$:
\[\lambda_{\max}(I-B^{-1}\E{Z}) = \|I-B^{-1}\E{Z}\|_B = \rho.\]
Therefore,
 \begin{equation} \label{eq:normEconv}
\norm{\E {x^{k} -x^{*}} }_B \leq\rho^{k} \norm{x^{0} - x^{*}}_B.
 \end{equation}
\end{theorem}

\begin{proof}
Taking expectations conditioned on $x^k$ in~\eqref{eq:xZupdate}, we get
\begin{equation}\label{eq:0suj9sj}\E{x^{k+1}-x^* \;|\; x^k} = (I-B^{-1}\E{Z})(x^k-x^*).\end{equation}
Taking expectation again gives
\begin{eqnarray*}
\E{x^{k+1}-x^*} &=& \E{\E{x^{k+1}-x^* \;|\; x^k}} \\
&\overset{\eqref{eq:0suj9sj}}{=}& \E{(I-B^{-1}\E{Z})(x^k-x^*)} \\ &=& (I-B^{-1}\E{Z})\E{x^k-x^*}.\end{eqnarray*}
Applying the norms to both sides we obtain the estimate
\[\norm{\E {x^{k+1} -x^{*}} }_B \leq \norm{I-B^{-1}\E{Z}}_B \, \norm{\E {x^{k} -x^{*}} }_B. \]
It remains to prove that $\rho = \norm{I-B^{-1}\E{Z}}_B$ and then unroll the recurrence.
According to the definition of operator norm~\eqref{eq:specnorm}, we have
\begin{align*}
\norm{I-B^{-1}\E{Z}}_B^2 &=  \max_{\norm{B^{1/2}x}_2=1} \norm{B^{1/2}(I-B^{-1}\E{Z})x }_2^2.
\end{align*}																		
Substituting $B^{1/2}x =y$ in the above gives
\begin{align*}
\norm{I-B^{-1}\E{Z}}_B^2 &= \max_{\norm{y}_2=1} \norm{B^{1/2}(I-B^{-1}\E{Z})B^{-1/2}y}_2^2 \nonumber \\ 
&= \max_{\norm{y}_2=1} \norm{(I-B^{-1/2}\E{Z}B^{-1/2})y}_2^2 \nonumber\\
&= \lambda_{\max}^2 (I-B^{-1/2}\E{Z}B^{-1/2})\\
&=  \left(1-\lambda_{\min}(B^{-1/2}\E{Z}B^{-1/2})\right)^2  = \rho^2,\nonumber
\end{align*}
where in the third equality we used the symmetry of $(I-B^{-1}\E{Z}B^{-1})$ when passing from the operator norm
to the spectral radius. Note that the symmetry of $\E{Z}$ derives from the symmetry of $Z$. 
\end{proof}


\subsection{Expectation of  Norm}
We now turn to analysing the convergence of the expected norm of the error, for which we need the following technical lemma.

\begin{lemma} \label{lem:y}
If $\E{Z}$ is positive definite,  then 
\begin{equation}\label{eq:EZyy} \dotprod{\E{Z}y,y} \geq (1-\rho)\norm{y}_{B}^2, \quad \forall y \in \R^n. \end{equation}
\end{lemma}
\begin{proof} As $\E{Z}$ and $B$ are positive definite, we get
\begin{align*}
1-\rho = \lambda_{\min}(B^{-1/2}\E{Z}B^{-1/2})
&= \max_{t} \left\{ t \quad | \quad B^{-1/2}\E{Z}B^{-1/2} -t \cdot I \succeq 0\right\}\\
&=\max_{t} \left\{ t \quad | \quad \E{Z}-t\cdot B \succeq 0\right\}.
\end{align*}
Therefore, $\E{Z}\succeq (1-\rho)B$, and the result follows.\end{proof}
  
\begin{theorem}[Expectation of norm]\label{theo:Enormconv} If $\E{Z}$ is positive definite, then
\begin{equation} \label{eq:Enormconv}
 \E {\|x^{k} -x^{*} \|_B^2 } \leq \rho^k \norm{x^{0} - x^{*}}_B^2,
 \end{equation}
 where $\rho<1$ is given in \eqref{eq:rho}.
\end{theorem}
\begin{proof}
Let $r^k = x^k-x^*$. Taking the expectation of~\eqref{eq:xZupdate} conditioned on $r^k$ we get
\begin{eqnarray*}
\E{\|r^{k+1}\|_B^2 \, | \, r^k} &\overset{\eqref{eq:xZupdate}}{=} &\E{\|(I-B^{-1}Z)r^k\|_B^2 \, | \, r^k } \nonumber \\
&\overset{\eqref{eq:projectionxx}}{=} & \E{\left< (B-Z)r^k,r^k\right> \, | \, r^k} \nonumber\\
& = & \|r^k\|_B^2 - \left< \E{Z} r^k,r^k\right> \quad  \overset{(\text{Lemma}~\eqref{lem:y})}{\leq} \quad \rho \cdot \|r^{k}\|^2_B.
\end{eqnarray*}
Taking expectation again and unrolling the recurrence gives the result. 
\end{proof}

The convergence rate $\rho$ of the expected norm of the error is ``worse'' than the $\rho^2$ rate of convergence of the norm of the expected error in Theorem~\ref{theo:normEconv}.
 This should not be misconstrued as Theorem~\ref{theo:normEconv} offering a ``better'' convergence rate than Theorem~\ref{theo:Enormconv}, because, as explained in Lemma~\ref{lem:conv}, convergence of the expected norm of the error is a stronger type of convergence. More importantly, the exponent is not of any crucial importance; clearly, an exponent of $2$ manifests itself only in halving the number of iterations.

\section{Methods Based on Discrete Sampling}\label{sec:discrete}

When $S$ has a discrete distribution, we can establish under reasonable assumptions when $\E{Z}$ is positive definite (Proposition~\ref{pro:Ediscrete}),  we can optimize the convergence rate in terms of the chosen probability distribution, and finally, determine a  probability distribution for which the convergence rate is expressed in terms of the scaled condition number (Theorem~\ref{theo:convsingleS}).

\begin{assumption}[Complete Discrete Sampling]
The random matrix $S$ has a discrete distribution. In particular,  $S= S_i \in \R^{m \times q_i}$ with probability  $p_i>0$, where $S_i^TA$ has full row rank and $q_i \in \N,$ for $i=1,\ldots, r$. Furthermore $\mathbf{S} \eqdef [S_1, \ldots, S_r] \in \R^{\rob{m}\times \sum_{i=1}^r q_i}$
is such that $A^T\mathbf{S}$ has full row rank.
\end{assumption}

\rob{As an example of complete discrete sampling, 
if $A$ has full column rank  and each row of $A$ is not strictly zero, $S =e^i$ with probability $p_i =1/n$, for $i =1,\ldots, n,$ then $\mathbf{S} =I$ then $S$ is a complete discrete sampling.
In fact, from any basis of $\R^n$ we can construct a complete discrete sampling in an analogous way. 

When $S$ is a complete discrete sampling, then $S^TA$ has full row rank and 
$(S^T A B^{-1}A^T S)^{\dagger}=(S^T A B^{-1}A^T S)^{-1}.$ Therefore we replace the pseudo-inverse in~\eqref{eq:MP} and~\eqref{eq:xZupdate} by the inverse.
} \rob{Furthermore, }using a complete discrete sampling guarantees convergence of the resulting method.

\begin{proposition}\label{pro:Ediscrete} 
Let $S$ be a complete discrete sampling,  then $\E{Z}$ is positive definite.
\end{proposition}
\begin{proof}
Let 
\begin{equation}\label{eq:defD}
D \eqdef	 \mbox{diag}\left( \sqrt{p_1}((S_1)^T A B^{-1}A^T S_1)^{-1/2}, \ldots, 
\sqrt{p_r}((S_r)^T A B^{-1}A^T S_r)^{-1/2}\right)
\end{equation}
which is a block diagonal matrix, and is well defined and invertible as $S_i^T A$  has full row rank for $i=1,\ldots, r$.
 Taking the expectation of $Z$~\eqref{eq:Z-first} gives
\begin{align}
\E{Z} &= \sum_{i=1}^r A^T S_i (S_i^T A B^{-1}A^T S_i)^{-1}S_i^T A p_i \nonumber \\
&=  A^T\left(\sum_{i=1}^r   S_i \sqrt{p_i}(S_i^T A B^{-1}A^T S_i)^{-1/2} (S_i^T A B^{-1}A^T S_i)^{-1/2}  \sqrt{p_i}S_i^T \right) A \nonumber \\
&= \left(  A^T \mathbf{S} D\right)  \left( D \mathbf{S}^T  A\right), \label{eq:EZdiscrete}
\end{align} which is positive definite because $A^T\mathbf{S}$ has full row rank and $D$ is invertible. 
\end{proof}\\
 With $\E{Z}$ positive definite, we can apply the convergence Theorem~\ref{theo:normEconv} and~\ref{theo:Enormconv}, and the resulting method converges.
 
 
 \subsection{Optimal Probabilities}
 
We can choose the discrete probability distribution that
 optimizes the convergence rate. For this, according to Theorems~\ref{theo:Enormconv} and~\ref{theo:normEconv} we need to find $p=(p_1,\dots,p_r)$ that maximizes the minimal eigenvalue of $B^{-1/2}\E{Z}B^{-1/2}$.
Let $S$ be a complete discrete sampling and fix the sample matrices $S_1,\dots, S_r$. Let us denote $Z=Z(p)$ as a function of $p=(p_1,\dots,p_r)$.  Then we can also think of the spectral radius as a function of $p$ where \[\rho(p) = 1 - \lambda_{\min}(B^{-1/2}\E{Z(p)}B^{-1/2}).\]

 Letting \[\Delta_r \eqdef \left\{p = (p_1,\dots,p_r) \in \R^r \;:\; \sum_{i=1}^r p_i =1, \; p\geq 0\right\},\] the problem of minimizing the spectral radius (i.e., optimizing the convergence rate) can be written as
 \[\rho^* \quad \eqdef\quad \min_{p\in \Delta_r} \rho(p) \quad = \quad 1 - \max_{p\in \Delta_r} \lambda_{\min} (B^{-1/2}\E{Z(p)}B^{-1/2}).\]
 This can be cast as a convex optimization problem, by first re-writing 
\begin{align*}
B^{-1/2}\E{Z(p)}B^{-1/2} &= \sum_{i=1}^r p_i \left(B^{-1/2}A^T S_i (S_i^T A B^{-1}A^T S_i)^{-1}S_i^T AB^{-1/2}  \right)\\
	 &= \sum_{i=1}^r p_i \left(V_i (V_i^T V_i)^{-1}V_i^T  \right),
\end{align*} 
 where $V_i = B^{-1/2}A^T S_i.$ Thus
\begin{equation}\label{eq:opt_sampling}\rho^* \quad = \quad 1 - \max_{p\in \Delta_r} \lambda_{\min}  \left( \sum_{i=1}^r p_i V_i (V_i^TV_i)^{-1} V_i^T  \right).\end{equation} 
 To obtain $p$ that maximizes the smallest eigenvalue, we solve 
 \begin{align}
 \max_{p,t} \,\, &\quad t  \nonumber \\
 \mbox{subject to}& \quad \sum_{i=1}^r p_i \left(V_i (V_i^T V_i)^{-1}V_i^T  \right) \succeq t\cdot  I, \label{eq:optconv}\\
 & \quad p \in \Delta_r. \nonumber
 \end{align}
 Despite~\eqref{eq:optconv} being a convex semi-definite program{\footnote{\rob{When preparing a revision of this paper, we have learned about the existence of prior work~\cite{Dai2014} where the authors have also characterized the probability distribution that optimizes the convergences rate of the RK method as the solution to an SDP.}}, which is apparently a harder problem than solving the original linear system, investing the time into solving~\eqref{eq:optconv} \rob{using a solver for convex conic programming such as \texttt{cvx}~\cite{cvx} can pay off,} as we show in Section~\ref{sec:numopt}. Though for a practical method based on this, we would need to develop an approximate solution to~\eqref{eq:optconv} which can be efficiently calculated. 
  
 \subsection{Convenient Probabilities}

Next we develop a choice of probability distribution that yields a convergence rate that is easy to interpret. This result is new and covers a wide range of methods,  including randomized Kaczmarz, randomized coordinate descent, as well as their block variants. However, it is more general, and covers many other possible particular algorithms, which arise by choosing a particular set of sample matrices $S_i$, for $i=1,\ldots, r.$

\begin{theorem} \label{theo:convsingleS}  Let $S$ be a complete discrete sampling such that $S = S_i \in \R^{m}$ with probability 
\begin{equation}\label{eq:convprob}p_i~=~\dfrac{\Tr{S_i^T AB^{-1}A^TS_i}}{\norm{B^{-1/2}A^T\mathbf{S}}_F^2},\quad \mbox{for } \quad i=1,\ldots, { r}.
\end{equation}
Then the iterates~\eqref{eq:MP} satisfy 
\begin{equation} \label{eq:expnormcon}
\E{\norm{x^{k} -x^{*} }_B^2} \leq \rho_c^k \,\norm{x^{0} - x^{*}}_B^2,\end{equation}
where
\begin{equation}\label{eq:rhoconv} \rho_c = 1- \frac{\lambda_{\min}\left(\mathbf{S}^T A B^{-1}A^T \mathbf{S} \right)}{\norm{B^{-1/2}A^T\mathbf{S}}_F^2}.
\end{equation}
\end{theorem}
\begin{proof}
Let $t_i = \Tr{\rob{S_i}^T A B^{-1}A^T \rob{S_i}}$, and with~\eqref{eq:convprob} in~\eqref{eq:defD} we have
\[ D^2 =\frac{1}{\norm{B^{-1/2}A^T\mathbf{S}}_F^2}\mbox{diag}\left(t_1 (\rob{S_1}^T A B^{-1}A^T \rob{S_1})^{-1}, \ldots,
t_r (\rob{S_r}^T A B^{-1}A^T \rob{S_r})^{-1}\right),
 \]
thus
\begin{equation}\label{eq:Dlambda} \lambda_{\min}(D^2) = \frac{1}{\norm{B^{-1/2}A^T\mathbf{S}}_F^2}\min_i\left\{ \frac{t_i}{\lambda_{\max}(\rob{S_i}^T A B^{-1}A^T \rob{S_i})} \right\} \geq \frac{1}{\norm{B^{-1/2}A^T\mathbf{S}}_F^2}. 
\end{equation}
Applying the above in~\eqref{eq:EZdiscrete} gives
\begin{align}
\lambda_{\min}\left(B^{-1/2}\E{Z}B^{-1/2} \right)
&= \lambda_{\min}\left(B^{-1/2}  A^T \mathbf{S} D^2 \mathbf{S}^T  A B^{-1/2}\right) \nonumber\\
&= \lambda_{\min}\left(  \mathbf{S}^T  A B^{-1} A^T \mathbf{S} D^2\right) \nonumber\\
& \geq  \lambda_{\min}\left(\mathbf{S}^T  A B^{-1} A^T \mathbf{S}  \right)\lambda_{\min}(D^2) \label{eq:prodpdeig}\\
& \geq  \frac{\lambda_{\min}\left(\mathbf{S}^T  A B^{-1} A^T \mathbf{S}  \right)}{\norm{B^{-1/2}A^T\mathbf{S}}_F^2}, \nonumber
\end{align} 
where we used that if $B,C \in \R^{n\times n}$ are positive definite $\lambda_{\min}(BC) \geq \lambda_{\min}(B)\lambda_{\min}(C).$ Finally
\begin{equation}\label{eq:convrhobound}
1 - \lambda_{\min}\left(B^{-1/2}\E{Z}B^{-1/2} \right) \leq 1 - \frac{\lambda_{\min}\left(\mathbf{S}^T  A B^{-1} A^T \mathbf{S}  \right)}{\norm{B^{-1/2}A^T\mathbf{S}}_F^2}.\end{equation}
The result~\eqref{eq:expnormcon} follows by applying Theorem~\ref{theo:Enormconv}.
\end{proof}

The convergence rate $\lambda_{\min}\left(\mathbf{S}^T  A B^{-1} A^T \mathbf{S}\right)/ \norm{B^{-1/2}A^T\mathbf{S}}_F^2$ is known as the scaled condition number, and naturally appears in other numerical schemes, such as matrix inversion~\cite{Edelman1992,Demmel1988}. When $\rob{S_i =s_i} \in \R^n$ is a column vector then \[p_i~=~\left((\rob{s_i})^T AB^{-1}A^T \rob{s_i}\right)/\norm{B^{-1/2}A^T\mathbf{S}}_F^2,\] for $i=1,\ldots r.$ In this case, the bound~\eqref{eq:Dlambda} is an equality and $D^2$ is a scaled identity, so~\eqref{eq:prodpdeig} and consequently~\eqref{eq:convrhobound} are equalities. For block methods, it is different story, and there is much more slack in the inequality~\eqref{eq:convrhobound}. So much so, 
the convergence rate~\eqref{eq:rhoconv} does not indicate any advantage of using a block method (contrary to numerical experiments).
  To see the advantage of a block method, we need to use the exact expression for $\lambda_{\min}(D^2)$ given in~\eqref{eq:Dlambda}. Though this results in a somewhat harder to interpret convergence rate, a matrix paving could be used explore this block convergence rate,  as was done for the block Kaczmarz method~\cite{Needell2014,Needell2012a}.

By appropriately choosing $B$ and $S$, this theorem applied to RK method~\eqref{eq:RKintro}, the CD-LS method~\eqref{eq:CDLSintro} and the CD-pd method~\eqref{eq:CDpdintro},  yields the convergence results~\eqref{eq:RKconv},~\eqref{eq:CDLSconv} and~\eqref{eq:CDposconv}, respectively, for single column sampling or block methods alike.

This theorem also suggests a preconditioning strategy, in that, a faster convergence rate will be attained if $\mathbf{S}$ is an approximate inverse of $B^{-1/2}A^T.$ For instance, in the RK method where $B=I$, this suggests that an accelerated convergence can be attained if $S$ is a random sampling of the rows of a preconditioner (approximate inverse) of $A.$


 \section{Methods Based on Gaussian Sampling} \label{sec:gauss}

In this section we shall describe variants of our method in the case when $S$ is a Gaussian  vector with mean $0\in \R^m$ and  a positive definite covariance matrix $\Sigma\in \R^{m \times m}$.  That is, $S =\zeta \sim N(0,\Sigma)$. This applied to~\eqref{eq:MP} results in iterations of the form
\begin{equation}\label{eq:gaussupdate}
\boxed{x^{k+1} = x^{k}  - \frac{\zeta ^T(A x^{k}-b)}{\zeta^TAB^{-1}A^T\zeta } B^{-1}A^T\zeta}
\end{equation}
Due to the symmetry of the multivariate normal distribution, there is a zero probability that $\zeta \in \Null{A^T}$ for any nonzero matrix $A$. 


Unlike the discrete methods in Section~\ref{sec:examples},  to calculate an iteration of~\eqref{eq:gaussupdate} we need to compute the product of a matrix with a dense vector $\zeta$.
This significantly raises the cost of an iteration. Though in our numeric tests in Section~\ref{sec:numerics}, the faster convergence of the Gaussian method often pays off for their high iteration cost. 

To analyze the complexity of the resulting method let $\xi \eqdef B^{-1/2}A^T S,$
which is also Gaussian, distributed as $\xi\sim N(0, \Omega)$,
where $\Omega\eqdef B^{-1/2}A^T \Sigma A B^{-1/2}.$
In this section we  assume {{\em $A$ has full column rank}, so that $\Omega$ is always positive definite. \rob{The complexity of the method can be established through
\begin{eqnarray} 
\rho &=& 1 - \lambda_{\min}\left(\E{ B^{-1/2} ZB^{-1/2}}\right) 
 =  1- \lambda_{\min}\left(\E{\frac{\xi\xi^T}{\norm{\xi}^2_2}}\right).
 \end{eqnarray}
 We can simplify the above by using the lower bound
 \[ \E{\frac{\xi\xi^T}{\norm{\xi}^2_2}} \succeq \frac{2}{\pi}\frac{\Omega}{\Tr{\Omega}},\] 
which is proven in Lemma~\ref{lem:gaussdiag} in the Appendix. Thus  
 \begin{equation}\label{eq:rhoboundgauss} 1- \frac{1}{n} \leq \rho \leq 1 -\frac{2}{\pi}\frac{\lambda_{\min}(\Omega)}{\Tr{\Omega}},
 \end{equation}
 where we used the general lower bound in~\eqref{eq:rho_lower}.
 Lemma~\ref{lem:gaussdiag} also shows that $\E{\xi\xi^T/\norm{\xi}^2_2}$ is positive definite,  thus Theorem~\ref{theo:Enormconv} guarantees that the expected norm of the error of all Gaussian methods converges exponentially to zero. This bound is tight upto a constant factor. For illustration of this,  in the setting  with $A=I=\Sigma$ we have $\xi \sim N(0,I)$ and $\E{\xi\xi^T/\norm{\xi}^2_2} = \tfrac{1}{n} I,$  which yields \[1-\dfrac{1}{n}\leq \rho  \leq 1- \dfrac{2}{\pi }\cdot \dfrac{1}{n}.\]}

 When $n=2$, then in  
Lemma~\ref{lem:2Dgausscov} of the Appendix we prove that
\[\E{\frac{\xi\xi^T}{\norm{\xi}^2_2}} = \frac{\Omega^{1/2}}{\Tr{\Omega^{1/2}}},\]
which yields a very favourable convergence rate. 


\subsection{Gaussian Kaczmarz}\label{sec:gaussI}
Let $B=I$ and choose $\Sigma =I$ so that $S=\eta\sim N(0,I)$. Then~\eqref{eq:gaussupdate}  has the form
\begin{equation}\boxed{ x^{k+1} = x^{k} 
- \frac{\eta^T (Ax^{k}-b)}{\|A^T \eta\|_2^2} A^T \eta } \end{equation}
which we call the \emph{Gaussian Kaczmarz} (GK) method, for it is the analogous method to the Randomized Karcmarz method in the discrete setting.  Using the formulation~\eqref{eq:RF},  for instance, the GK method can be interpreted as
\[x^{k+1} =\arg \min_{x\in \R^n} \norm{x- x^{*}}^2 \quad \mbox{ subject to } \quad  x = x^{k} + A^T\eta y, \quad y\in \R. \]
Thus at each iteration, a random normal Gaussian vector $\eta$ is drawn and a search direction is formed by $A^T\eta.$
 Then, starting from the previous iterate $x^{k}$,  an exact line search is performed over this search direction so that the euclidean distance from the optimal is minimized.

\subsection{Gaussian Least-Squares} \label{sec:gaussATA}
Let $B=A^TA$ and choose $S\sim N(0,\Sigma)$ with $\Sigma=AA^T$. It will be convenient to write $S=A\eta$, where $\eta\sim N(0,I)$. Then method \eqref{eq:gaussupdate} then has the form
\begin{equation}\boxed{x^{k+1}  = x^{k} - \frac{\eta^T A^T(Ax^{k} -b)}{\|A\eta\|_2^2} \eta}
 \end{equation}
 which we call the \emph{Gauss-LS} method. This method has a natural interpretation through formulation~\eqref{eq:RF} as
\[x^{k+1} =\arg \min_{x\in \R^n} \frac{1}{2}\|Ax-b\|_2^2 \quad \mbox{ subject to } \quad  x = x^{k} + y\eta, \quad y \in \R. \]
 That is, starting from $x^{k}$, we take a step in a random (Gaussian) direction, then perform an exact line search over this direction that minimizes the least squares error.
Thus the Gauss-LS method is the same as applying the Random Pursuit method~\cite{Stich2014} with exact line search to the Least-squares function.

\subsection{Gaussian Positive Definite}\label{sec:gaussA}
When $A$ is positive definite, we achieve an accelerated Gaussian method. 
Let $B =A$ and choose $S= \eta \sim N(0,I)$.  Method~\eqref{eq:gaussupdate} then has the form
\begin{equation}\label{eq:gausspd}\boxed{x^{k+1} 
= x^{k} - \frac{ \eta^T(A x^{k}-b)}{\norm{\eta}_A^2} \eta}
 \end{equation}
 which we call the \emph{Gauss-pd} method.
 
Using  formulation~\eqref{eq:RF}, the  method can be interpreted as
\[x^{k+1} =\arg \min_{x\in \R^n} \left\{f(x) \eqdef \tfrac{1}{2} x^TAx -b^T x\right\} \quad \mbox{ subject to } \quad  x = x^{k} + y\eta, \quad y \in \R. \]
 That is, starting from $x^{k}$, we take a step in a random (Gaussian) direction, then perform an exact line search over this direction. Thus the Gauss-pd method is equivalent to applying the Random Pursuit method~\cite{Stich2014} with exact line search to $f(x).$
 
      \rob{All the Gaussian methods can be extended to  block versions. We illustrate this by designing a Block Gauss-pd method where $S \in \R^{n \times q}$ has   i.i.d.\ Gaussian normal entries and $B=A.$ This results in the  iterates
 \begin{equation} \label{eq:Bgausspd} x^{k+1} 
= x^{k} - S(S^TAS)^{-1}S^T(A x^{k}-b).
 \end{equation}
 }

\section{Numerical Experiments} \label{sec:numerics}
 We perform some preliminary numeric tests.
   Everything was coded and run in MATLAB R2014b. 
Let \rob{$\kappa_2 = \norm{A}\norm{A^{\dagger}}$} be the $2-$norm condition number, where $A^{\dagger}$ is a pseudo-inverse of $A$. In comparing different methods for solving overdetermined systems, we use the relative error measure $\norm{Ax^k-b}_2/\norm{b}_2,$ while for positive definite systems we use $\norm{x^k-x^*}_A/\norm{x^*}_A$ as a relative error measure. We run each method until the relative error is below $10^{-4}$ or until $300$ seconds in time is exceeded. We use $x_0=0 \in \R^n$ as an initial point.
In each figure we plot the relative error in percentage \rob{on the vertical axis}, starting with $100\%$. \rob{For the horizontal axis, we use either wall-clock time measured using the \texttt{tic-toc} MATLAB function or the total number of floating point operations (\emph{flops}).}

In implementing the discrete sampling methods we used the convenient probability distributions~\eqref{eq:convprob}.
  
\rob{All tests were performed on a Desktop with 64bit quad-core Intel(R) Core(TM) i5-2400S CPU @2.50GHz with 6MB cache size with a Scientific Linux release 6.4 (Carbon) operating system.}
  
 \rob{Consistently across our experiments, the Gaussian methods almost always require more flops to reach a solution with the same precision as their discrete sampling counterparts. This is due to the expensive matrix-vector product required by the Gaussian methods. While the results are more mixed when measured in terms of wall clock time. This is because MATLAB  performs automatic multi-threading when calculating matrix-vector products, which was the bottleneck cost in the Gaussian methods. As our machine has four cores, this explains some of the difference observed when measuring performance in terms of  number of flops and wall clock time.} 
  
\subsection{Overdetermined linear systems}

First we compare the methods Gauss-LS, CD-LS, Gauss-Kaczmarz and RK methods on synthetic linear systems generated with the matrix functions {\tt rand} and {\tt sprandn}, see Figure~\ref{fig:oversynth}. 
The high iteration cost of the Gaussian methods resulted in poor performance on the dense problem generated using \texttt{rand} in Figure~\ref{fig:rand}. 
In Figure~\ref{fig:sprandn} we compare the methods on a sparse linear system generated  using the MATLAB sparse random matrix function \texttt{sprandn}($m,n$,{\tt density,rc}), where {\tt density} is the percentage of nonzero entries and {\tt rc} is the reciprocal of the condition number. On this sparse problem the Gaussian methods are more efficient, and converge at a similar rate to the discrete sampling methods.

\begin{figure}
    \centering
    \begin{subfigure}[t]{0.7\textwidth}
\includegraphics[width =  1.1\textwidth, height =4.5cm, trim= 32 270 47 285, clip ]{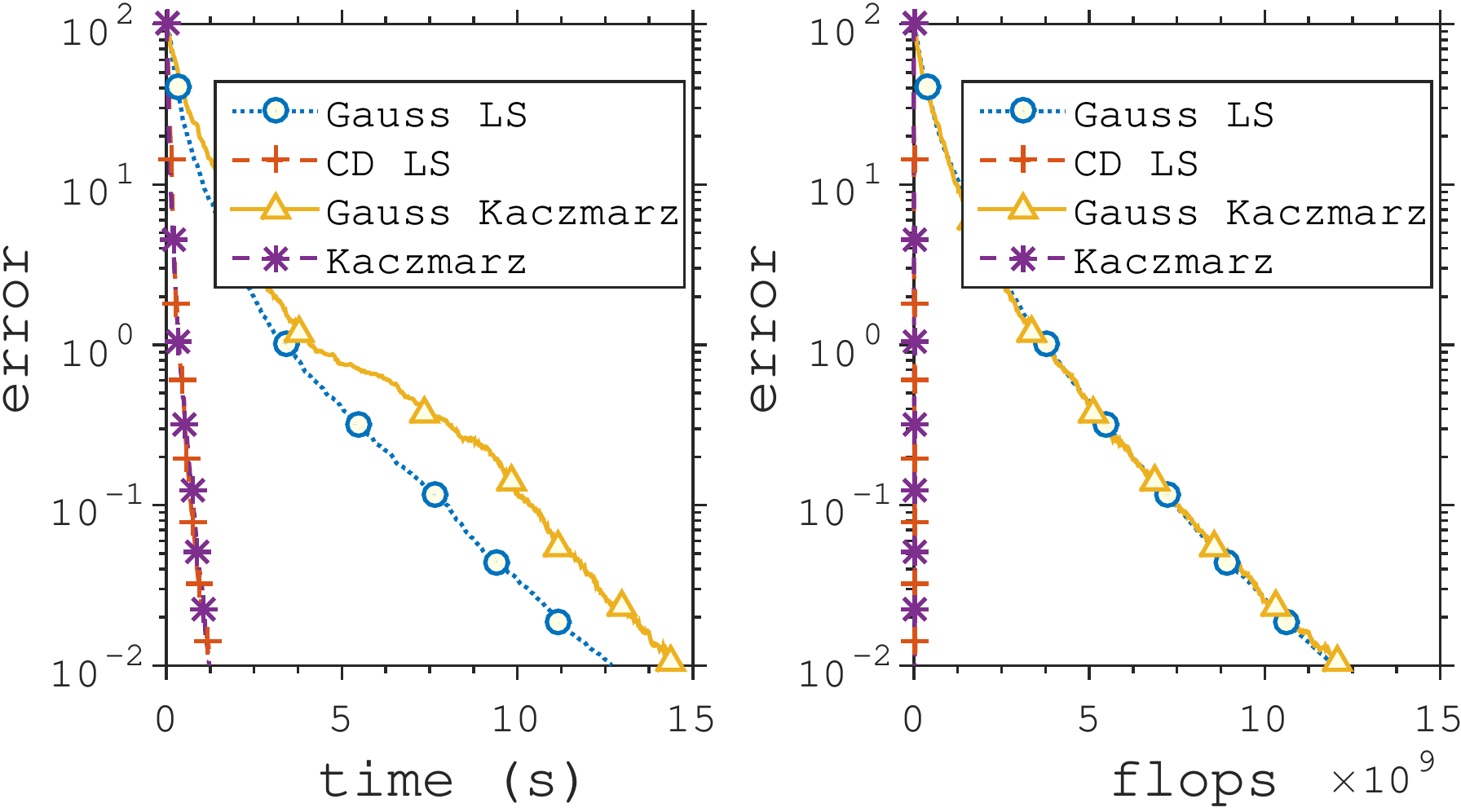}
        \caption{\texttt{rand}}\label{fig:rand}
    \end{subfigure}%
    \hspace{0.05\textwidth}
    \begin{subfigure}[t]{0.7\textwidth}
\includegraphics[width =  1.1\textwidth, height =4.5cm, trim= 32 270 47 285, clip ]{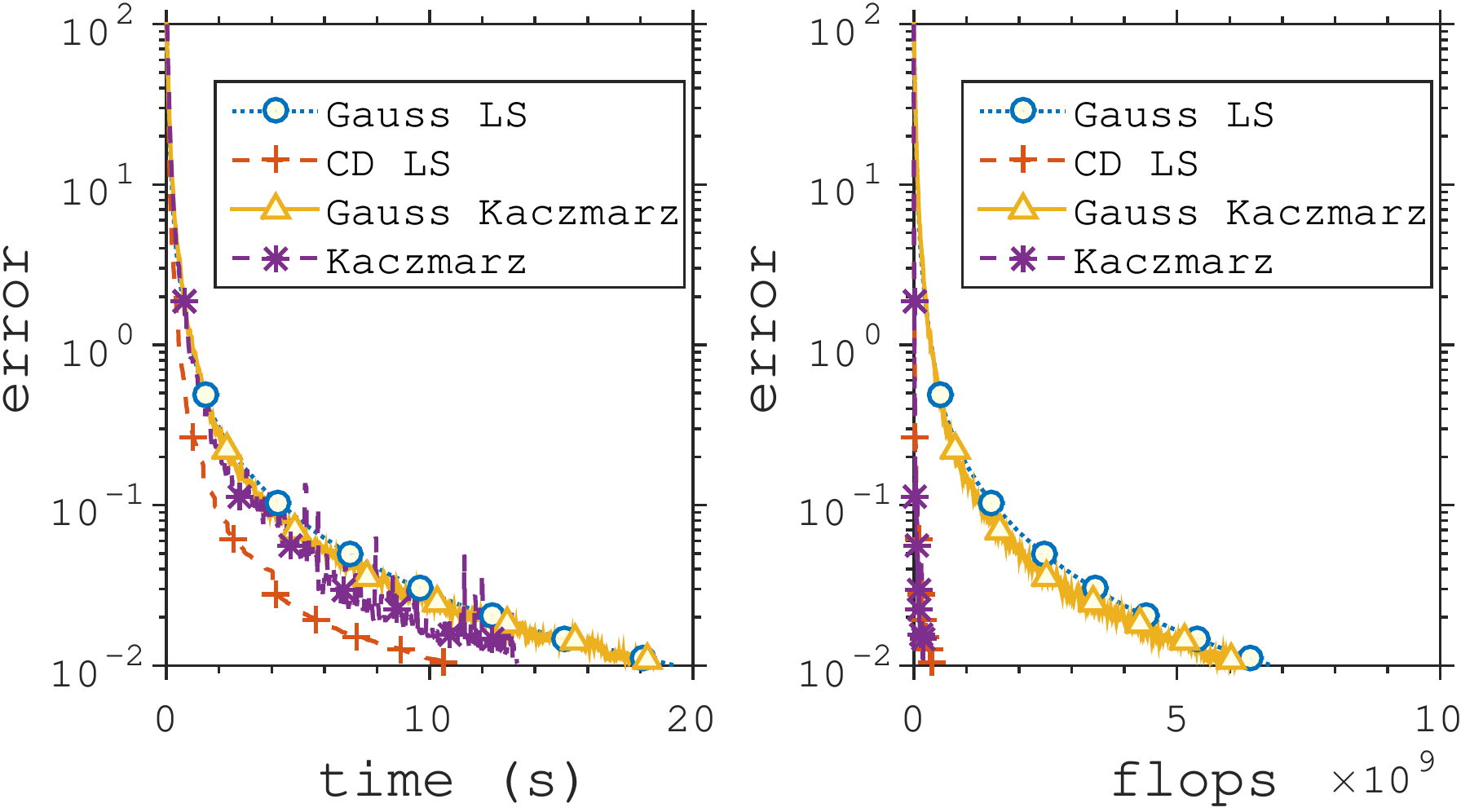}
        \caption{\texttt{sprandn}} \label{fig:sprandn}
    \end{subfigure}
    \caption{
The performance of the Gauss-LS, CD-LS, Gauss-Kaczmarz and RK methods on
    synthetic MATLAB generated problems (a) \texttt{rand}$(n,m)$ with $(m;n)=(1000,500)$ (b) \texttt{sprandn}($m,n$,{\tt density,rc}) with $(m;n) =(1000,500)$, {\tt density}$=1/\log(nm)$ and {\tt rc}$= 1/\sqrt{mn}$. In both experiments dense solutions were generated with $x^*=$\texttt{rand}$(n,1)$ and $b=Ax^*.$ }\label{fig:oversynth}
\end{figure}

In Figure~\ref{fig:overMM} we test two overdetermined linear systems taken from the the Matrix Market collection~\cite{Boisvert1997}. The collection also provides the right-hand side of the linear system.  Both of these systems are very well conditioned, but do not have full column rank, thus Theorem~\ref{theo:Enormconv} does not apply. The four methods have a similar performance on Figure~\ref{fig:illc1033}, while the  Gauss-LS and  CD-LS method converge faster on~\ref{fig:well1033} as compared to the Gauss-Kaczmarz and Kaczmarz methods.

\begin{figure}
    \centering
    \begin{subfigure}[t]{0.70\textwidth}
\includegraphics[width =  1.1\textwidth, height =4.5cm, trim= 32 270 47 285, clip ]{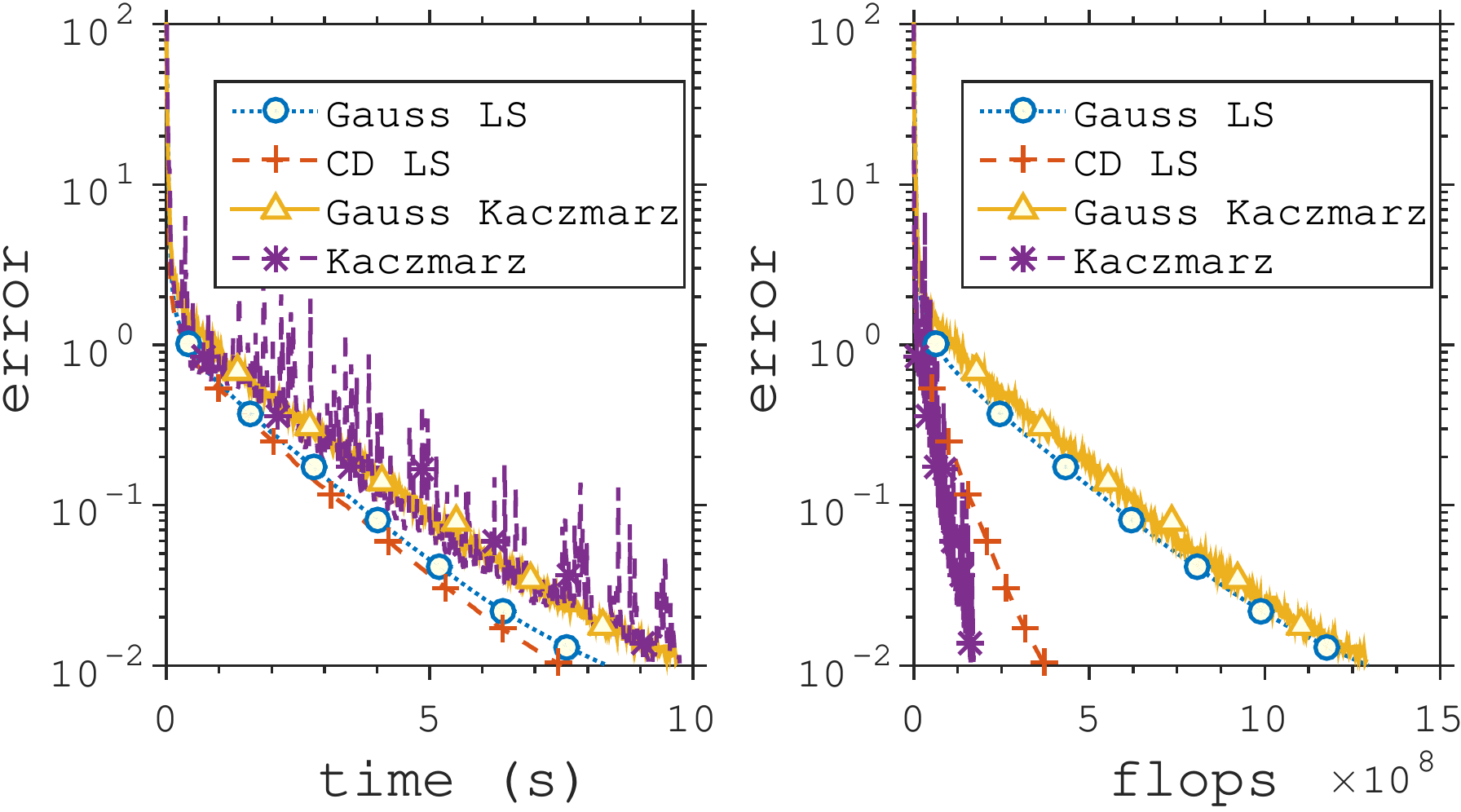}
        \caption{\texttt{illc1033}}\label{fig:illc1033}
    \end{subfigure}%
    \hspace{0.05\textwidth}
    \begin{subfigure}[t]{0.70\textwidth}
\includegraphics[width =  1.1\textwidth, height =4.5cm, trim= 32 270 47 285, clip ]{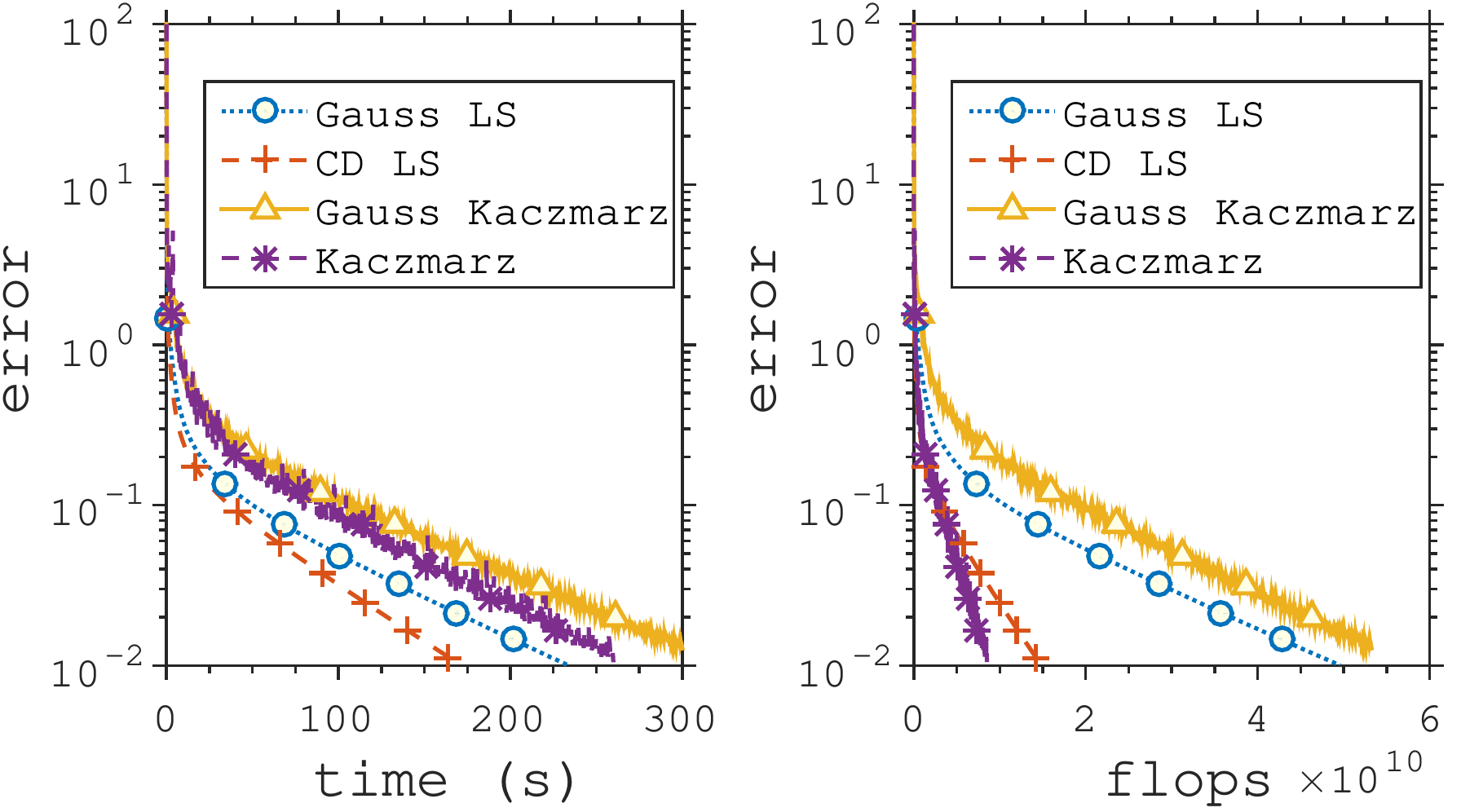}
        \caption{\texttt{well1033} } \label{fig:well1033}
    \end{subfigure}
    \caption{The performance of the Gauss-LS, CD-LS, Gauss-Kaczmarz and RK methods on
        linear systems (a) \texttt{well1033} where $(m;n) = (1850,750)$, $nnz = 8758$ and $\kappa_2 = 1.8$ (b) \texttt{illc1033} where $(m;n) =(1033;320)$, $nnz= 4732$ and $\kappa_2 = 2.1$, from the Matrix Market~\cite{Boisvert1997}.}\label{fig:overMM}
\end{figure}

Finally, we test  two problems, the \texttt{SUSY} problem  and the \texttt{covtype.binary} problem, from the library of support vector machine problems LIBSVM~\cite{Chang2011}. These problems do not form consistent linear systems, thus only the Gauss-LS and CD-LS methods are applicable, see Figure~\ref{fig:LIBSVM}.  This is equivalent to applying the Gauss-pd and CD-pd to the least squares system $A^TAx=A^Tb,$ which is always consistent.

\begin{figure}
    \centering
    \begin{subfigure}[t]{0.7\textwidth}
        \centering
\includegraphics[width =  1.1\textwidth, height =4.5cm, trim= 28 270 47 285, clip ]{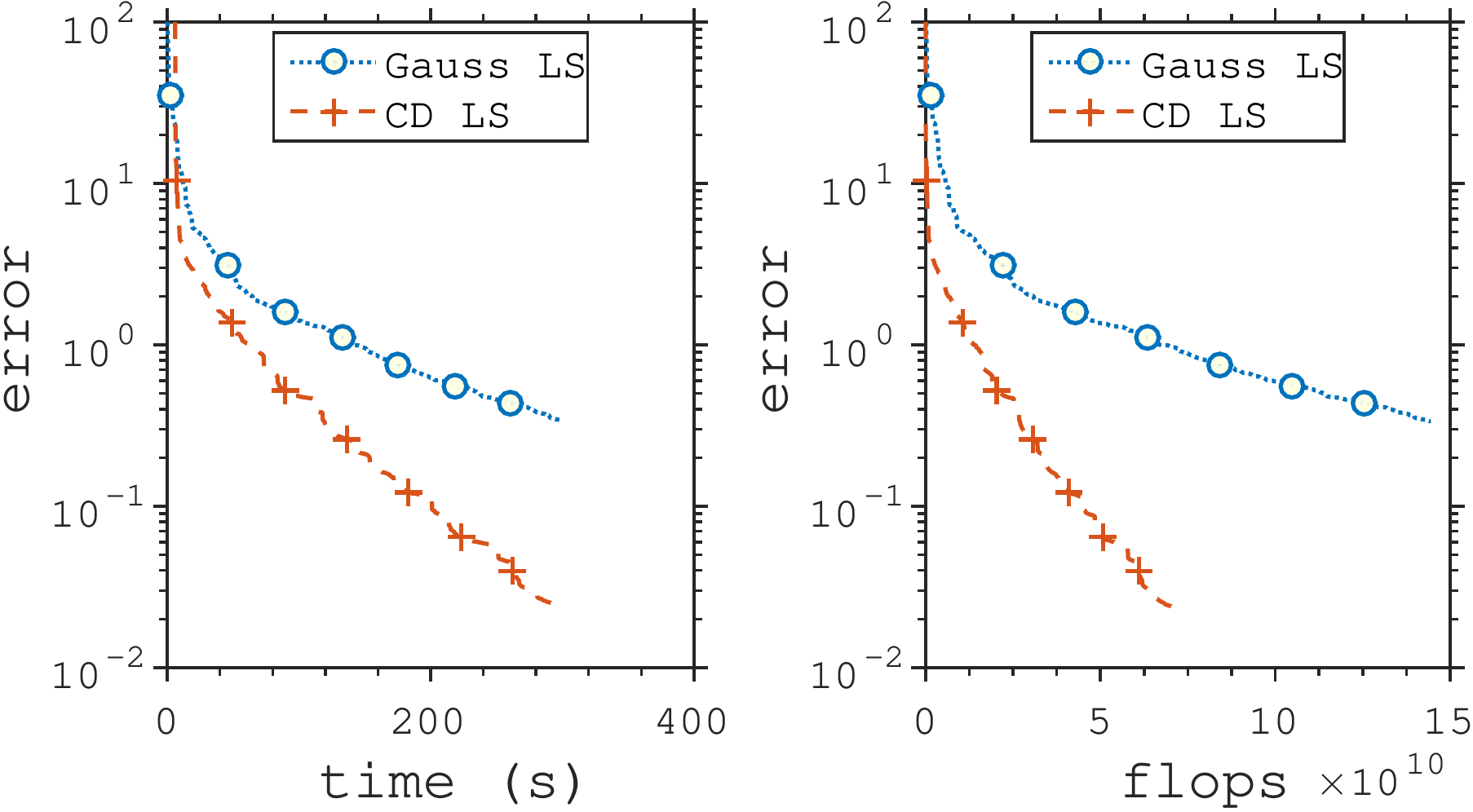}
        \caption{\texttt{SUSY}}
    \end{subfigure}%
    \hspace{0.05\textwidth}
    \begin{subfigure}[t]{0.7\textwidth}
        \centering
\includegraphics[width =  1.1\textwidth, height =4.5cm, trim= 30 270 47 285, clip ]{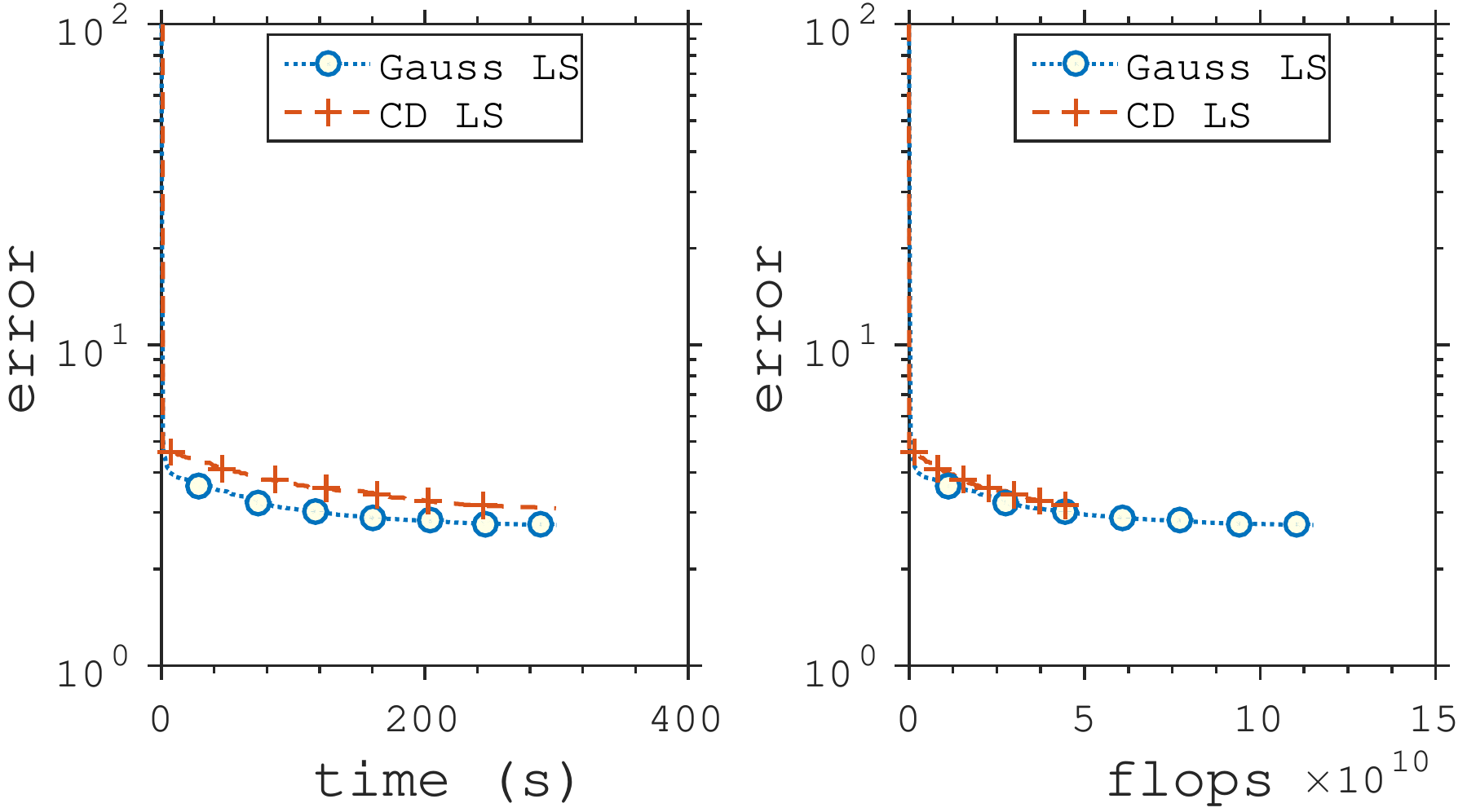}
        \caption{\texttt{covtype-libsvm-binary}}
    \end{subfigure}
    \caption{The performance of Gauss-LS and CD-LS methods on two LIBSVM test problems: (a)  \texttt{SUSY}: $(m;n)=(5\times 10^6; 18)$ (b) \texttt{covtype.binary}:  $(m;n)=(581,012; 54)$.} \label{fig:LIBSVM}
\end{figure}
       
\rob{Despite the higher iteration cost of the Gaussian methods, their performance, in terms of the wall-clock time, is comparable to performance of the discrete methods when the system matrix is sparse.}

\subsection{Bound for Gaussian convergence} \label{sec:numgaussbound}

\rob{Now we compare the error over the number iterations of the Gauss-LS method to theoretical rate of convergence given by the bound~\eqref{eq:rhoboundgauss}. For the Gauss-LS method~\eqref{eq:rhoboundgauss} becomes \[1-\frac{1}{n} \leq \rho \leq 1-\frac{2}{\pi}\lambda_{\min}\left( \frac{A^TA}{\norm{A}_F^2}\right).\]
 In Figures~\ref{fig:overtheouniform} and~\ref{fig:overtheoliver} we compare the empirical and theoretical bound on a random Gaussian matrix and the \texttt{liver-disorders} problem~~\cite{Chang2011}. Furthermore, we ran the Gauss-LS method 100 times and plot as dashed lines the 95\% and 5\% quantiles. These tests indicate that the bound it tight for well conditioned problems, such as Figure~\ref{fig:overtheouniform} in which the system matrix has a condition number equal to $1.94$. While in Figure~\ref{fig:overtheoliver} the system matrix has a condition number of $41.70$ and there is some much more slack between the empirical convergence and the theoretical bound. }

\begin{figure}
    \centering
    \begin{subfigure}[t]{0.47\textwidth}
\includegraphics[width =  \textwidth, trim= 110 280 100 285, clip ]{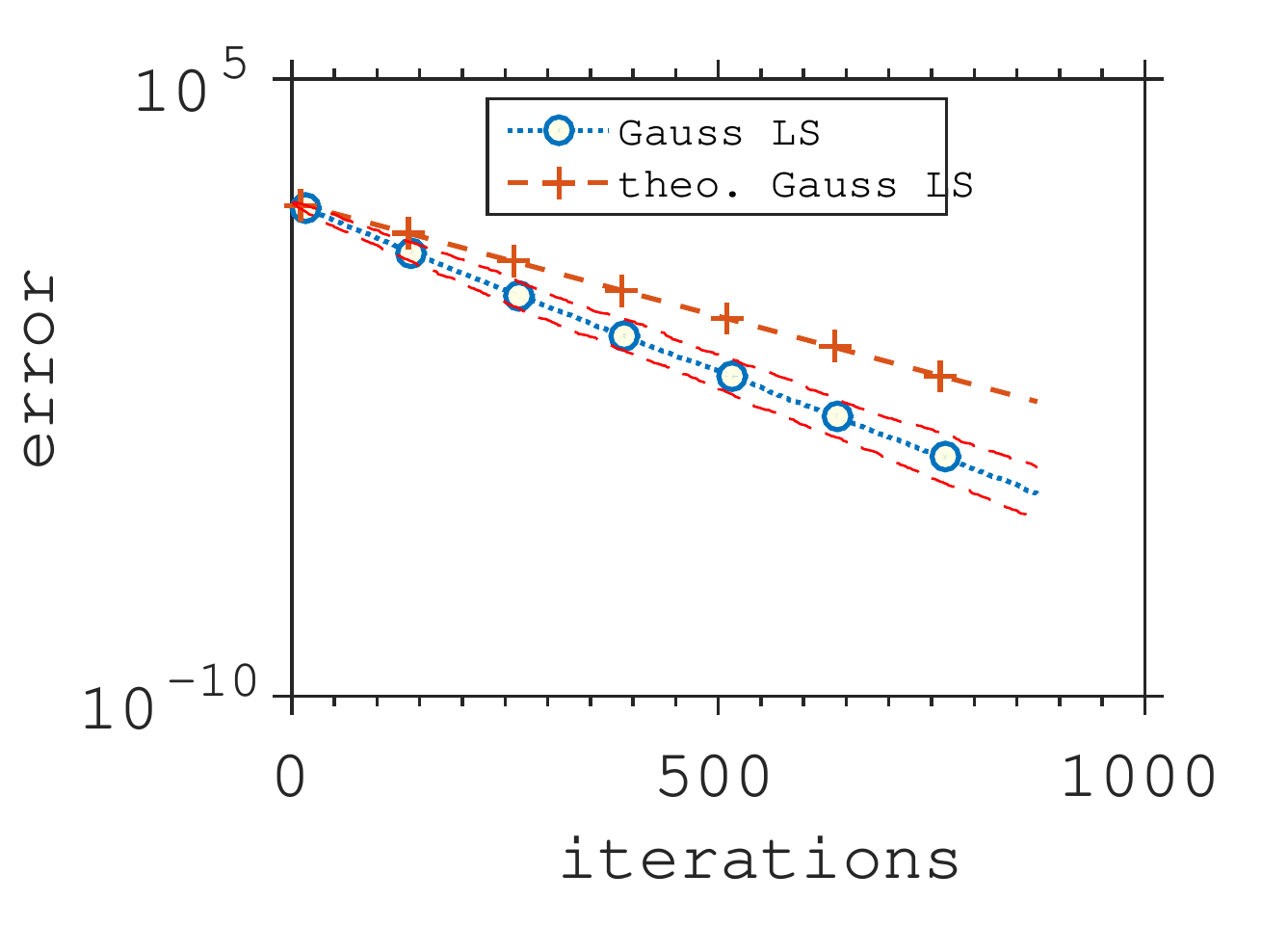}
        \caption{\texttt{rand}$(n,m)$}\label{fig:overtheouniform}
    \end{subfigure}%
    \hspace{0.05\textwidth}
    \begin{subfigure}[t]{0.47\textwidth}
\includegraphics[width =  \textwidth, trim= 110 280 100 285, clip ]{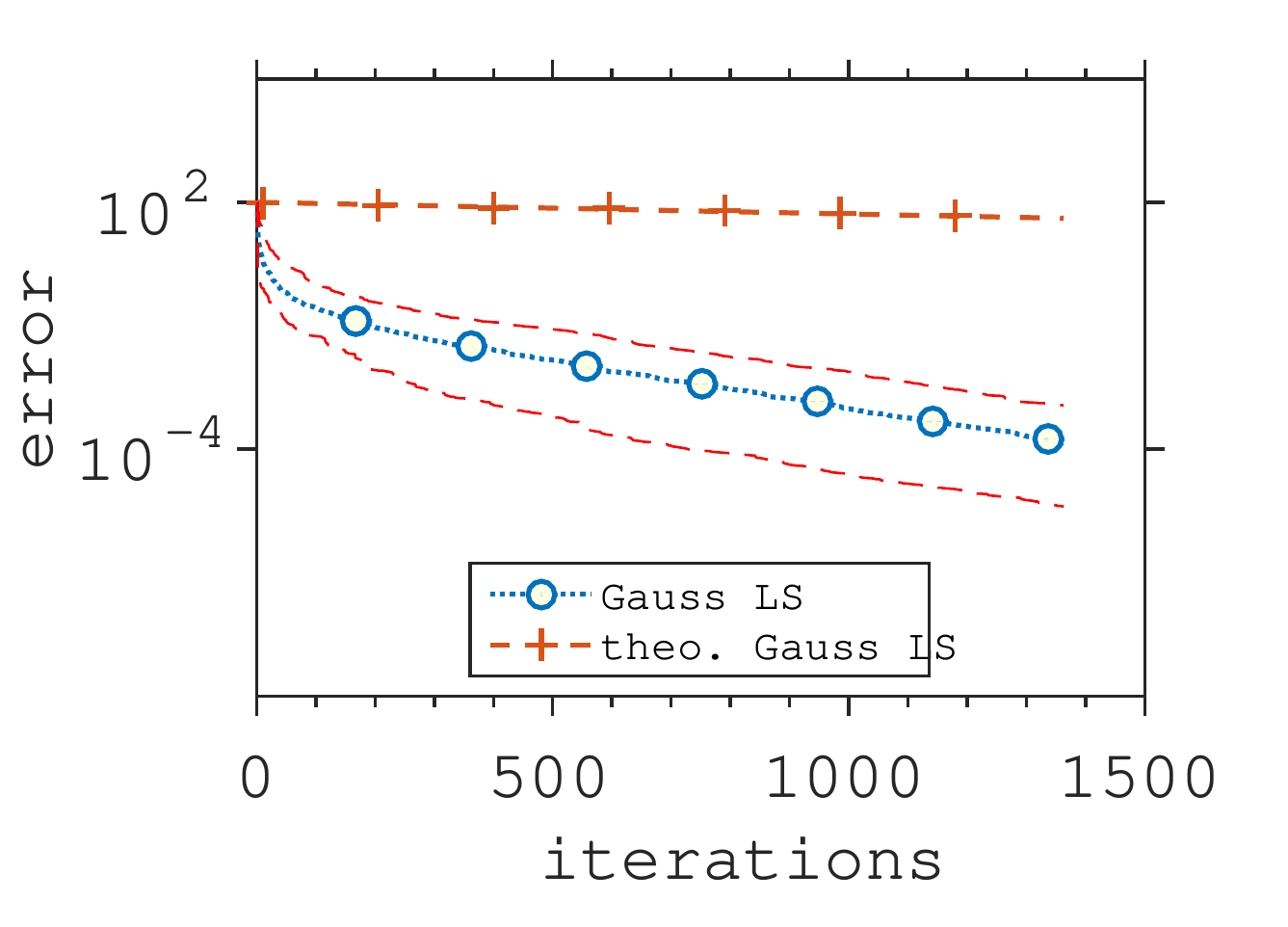}
        \caption{\texttt{liver-disorders}} \label{fig:overtheoliver}
    \end{subfigure}
    \caption{\rob{A comparison between the Gauss-LS method and the theoretical bound $\rho_{theo} \eqdef 1-\lambda_{\min}(A^TA)/\norm{A}_F^2$ on (a) \texttt{rand}$(n,m)$ with $(m;n)=(500,50), \kappa_2 =1.94$ and a dense solution generated with $x^*=$\texttt{ rand}$(n,1)$ (b) \texttt{liver-disorders} with $(m;n) =(345,6)$ and $\kappa_2 = 41.70.$} }\label{fig:overtheo}
\end{figure}       

\subsection{Positive Definite}
\rob{First we compare the two methods Gauss-pd~\eqref{eq:gausspd} and CD-pd~\eqref{eq:09j0s9jsss} on synthetic data in Figure~\ref{fig:2}.}
\begin{figure}
    \centering
    \begin{subfigure}[t]{0.47\textwidth}
\includegraphics[width =  \textwidth, trim= 30 270 40 285, clip ]{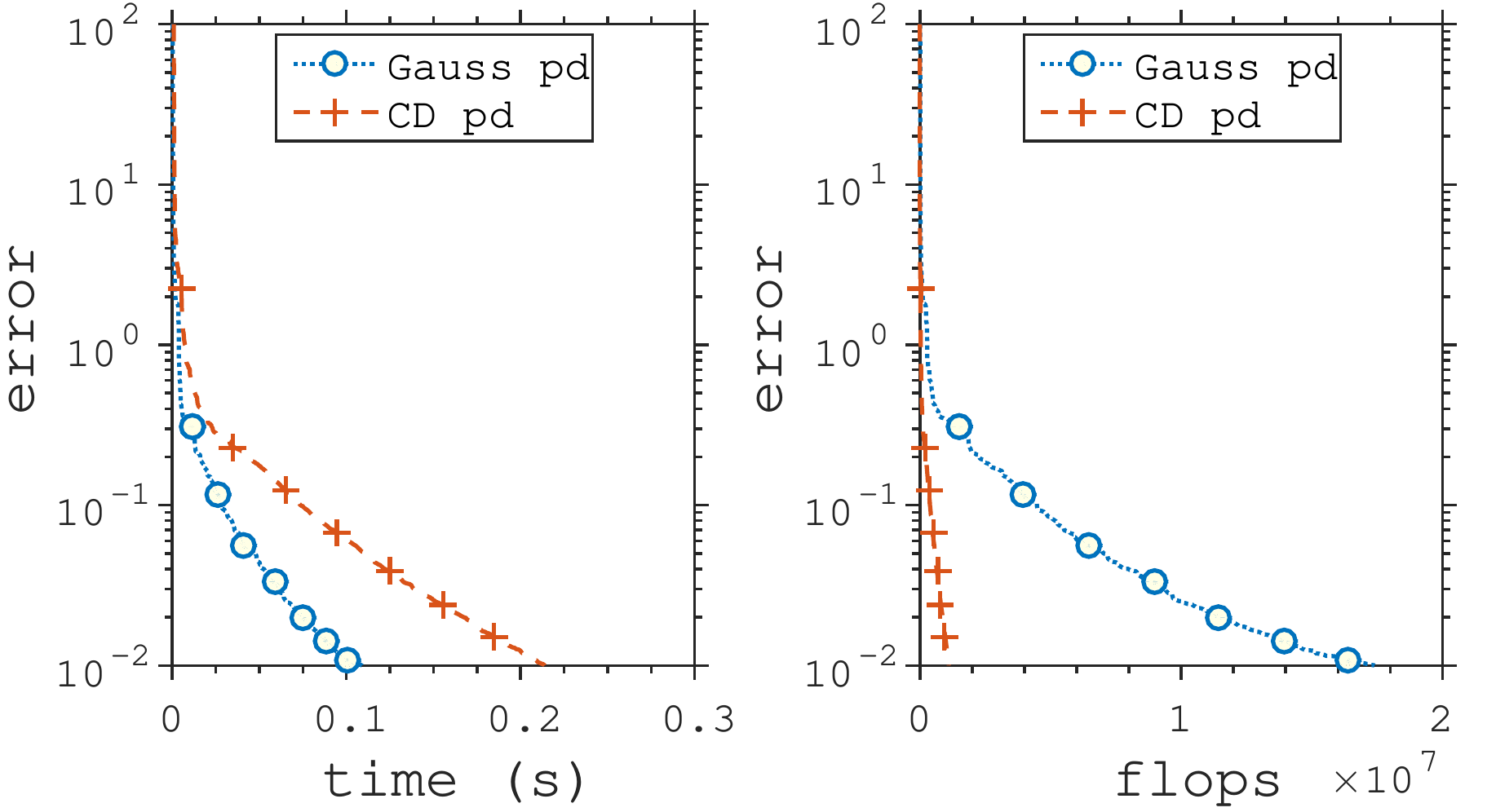}
    \end{subfigure}%
    \hspace{0.05\textwidth}
    \begin{subfigure}[t]{0.47\textwidth}
\includegraphics[width =  \textwidth, trim= 30 270 40 285, clip ]{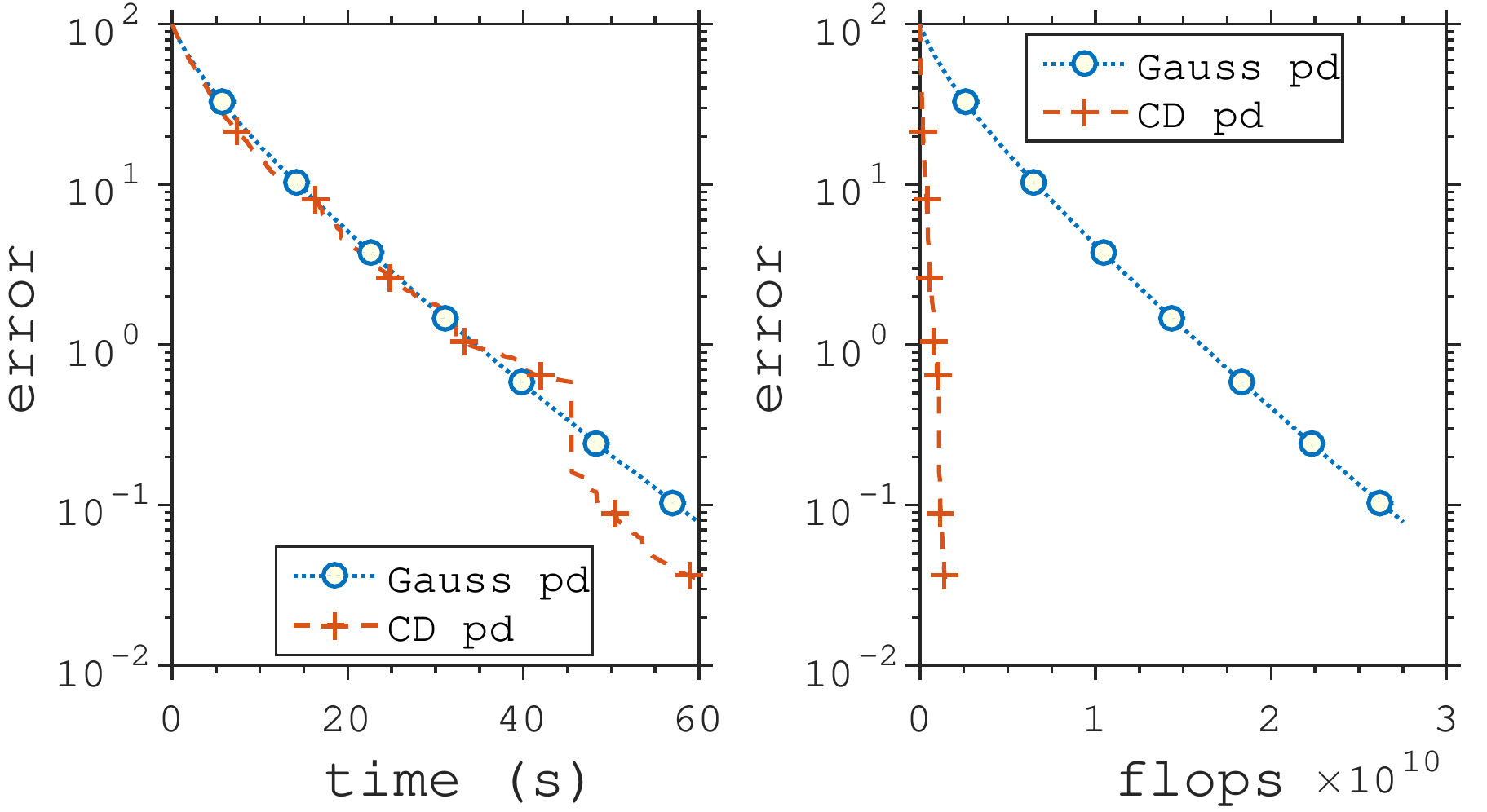}
    \end{subfigure}
    
\caption{Synthetic MATLAB generated problem. The Gaussian methods are more efficient on sparse matrices. LEFT: The \texttt{Hilbert Matrix} with $n=100$ and  condition number $\norm{A}\norm{A^{-1}}=6.5953\times10^{19}$. RIGHT: Sparse random matrix $A=$ \texttt{sprandsym} ($n$, {\tt density}, {\tt rc}, {\tt type}) with $n=1000$, {\tt density}$ =1/\log(n^2)$ and ${\tt rc}= 1/n=0.001$. Dense solution generated with $x^{*}=$\texttt{rand}$(n,1).$}\label{fig:2}
\end{figure}
Using the MATLAB function {\tt hilbert}, we can generate positive definite matrices with very high condition number, see Figure~\ref{fig:2}(LEFT). Both methods converge slowly and, despite the dense system matrix, the Gauss-pd method has a similar performance to CD-pd. In Figure~\eqref{fig:2}(RIGHT) we compare the two methods on a system generated by the  
 MATLAB function \texttt{sprandsym} ($m$, $n$, {\tt density}, {\tt rc}, {\tt type}), where {\tt density} is the percentage of nonzero entries, {\tt rc} is the reciprocal of the condition number and {\tt type=1} returns a positive definite matrix. \rob{The Gauss-pd  and the CD-pd method have a similar performance in terms of wall clock time on this sparse problem.}
\begin{figure}
    \centering
    \begin{subfigure}[t]{0.475\textwidth}
        \centering
\includegraphics[width =  \textwidth, trim= 30 270 40 285, clip ]{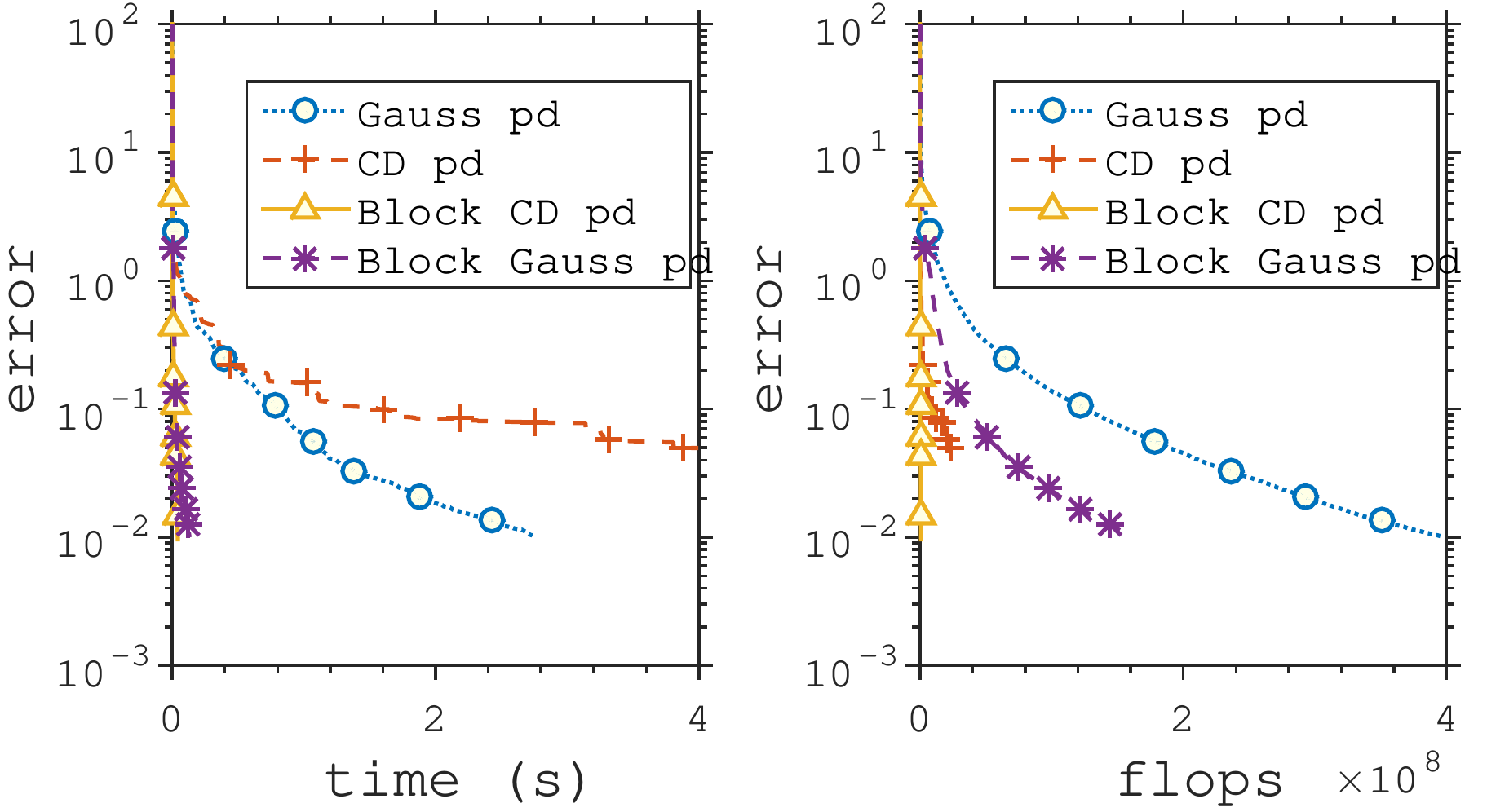}
        \caption{\texttt{aloi}}
    \end{subfigure}%
  \hspace{0.04\textwidth}
    \begin{subfigure}[t]{0.475\textwidth}
        \centering
\includegraphics[width =  \textwidth, trim= 30 270 40 285, clip ]{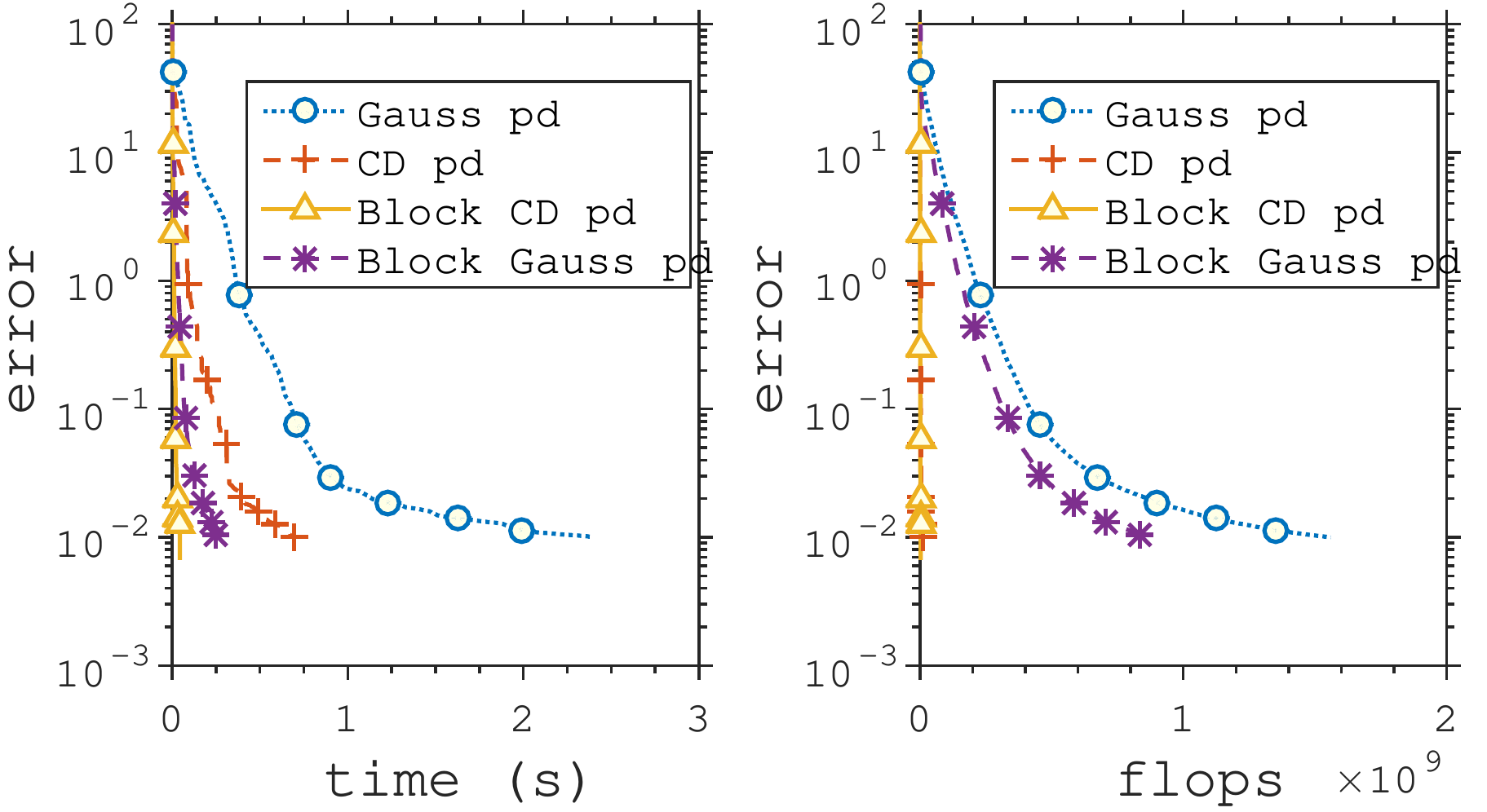}
        \caption{\texttt{protein}}
    \end{subfigure}
        \begin{subfigure}[t]{0.475\textwidth}
        \centering
\includegraphics[width =  \textwidth, trim=  30 270 40 285, clip ]{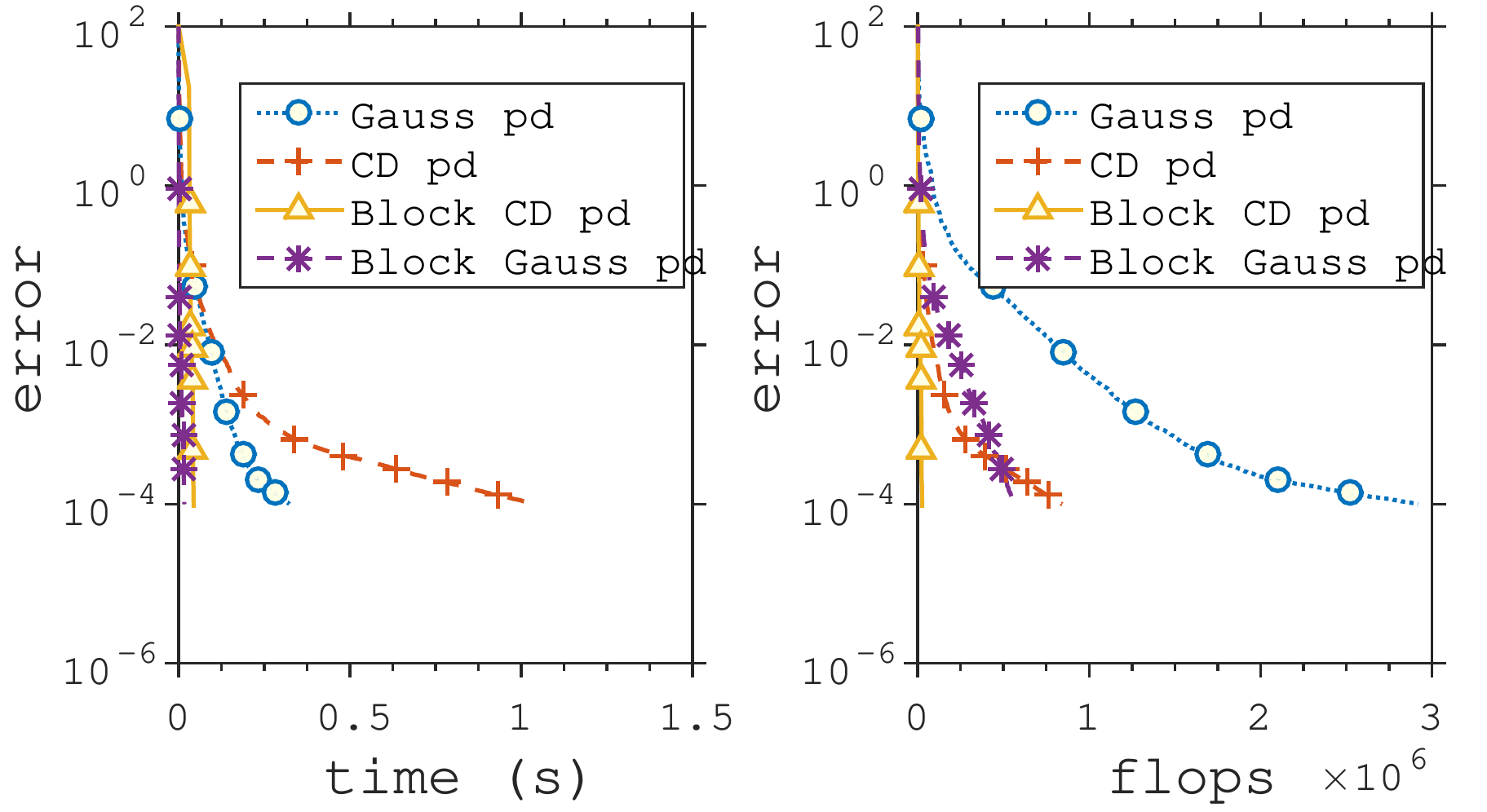}
        \caption{\texttt{SUSY}}
    \end{subfigure}%
  \hspace{0.04\textwidth}
    \begin{subfigure}[t]{0.475\textwidth}
        \centering
\includegraphics[width =  \textwidth, trim=  30 270 40 285, clip ]{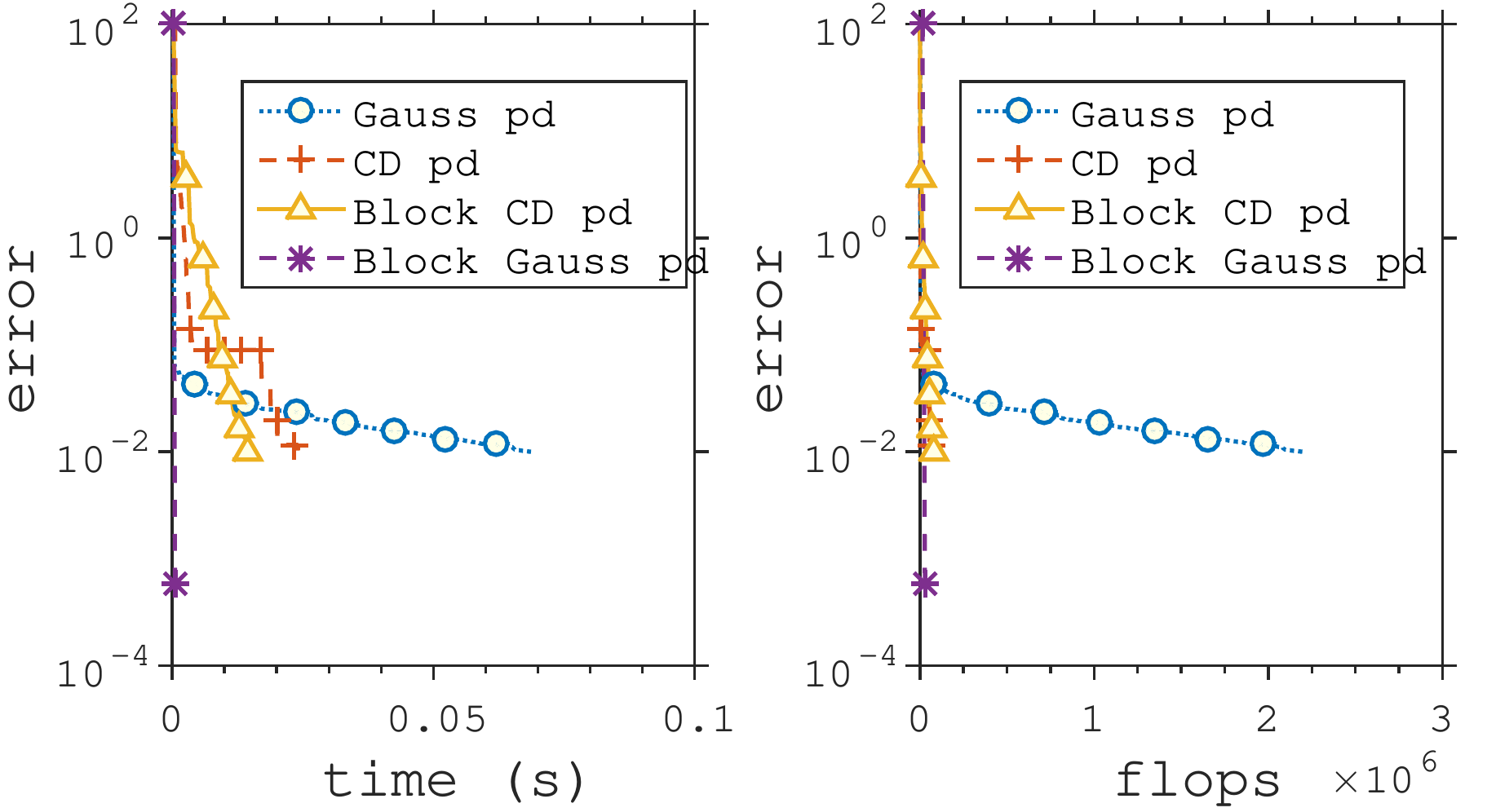}
        \caption{\texttt{covtype.binary}}
    \end{subfigure}
    \caption{The performance of Gaussian and Coordinate Descent pd methods on four ridge regression problems: (a) \texttt{aloi}: $(m;n)=(108,000;128)$ (b) \texttt{protein}: $(m; n)=(17,766; 357)$ (c)  \texttt{SUSY}: $(m;n)=(5\times 10^6; 18)$ (d) \texttt{covtype.binary}:  $(m;n)=(581,012; 54)$.} \label{fig:LIBSVMridge}
\end{figure}

\rob{ To appraise the performance gain in using block variants,
we perform tests using two block variants: the Randomized Newton method~\eqref{eq:CDpdblock},  which we will now refer to as the Block CD-pd method, and the Block Gauss-pd method~\eqref{eq:Bgausspd}. 
The size of blocks $q$ in both methods was set to $q = \sqrt{n}.$ To solve the $q\times q$ system required in the block methods, we use MATLAB's built-in direct solver, sometimes referred to as ``back-slash''. 

Next we test the Newton system $\nabla^2 f(w_0) x = - \nabla f(w_0)$,  arising from four ridge-regression problems of the form 
 \begin{equation}\label{eq:ridgeMatrix}
\min_{w\in \R^n}f(w)\eqdef \tfrac{1}{2} \norm{Aw-b}_2^2 + \tfrac{\lambda}{2} \norm{w}_2^2,
\end{equation}
using data from LIBSVM~\cite{Chang2011}. In particular, we set
$w_0=0$ and use $\lambda =1$ as the regularization parameter, whence $\nabla f(w_0) = A^Tb$ and $\nabla^2 f(w_0) = A^TA+ I$. 
}

  \rob{In terms of wall clock time, 
The Gauss-pd method converged faster on all problems accept the \texttt{protein} problem as compared to CD-pd. 
The two Block methods had a comparable performance on the \texttt{aloi} and the \texttt{SUSY} problem. 
The Block Gauss-pd method converged in one iteration on \texttt{covtype.binary}, and the Block CD-pd method converged fast on the \texttt{Protein} problem.  }

\rob{We now compare the methods on two positive definite matrices from the Matrix Market collection~\cite{Boisvert1997}, see Figure~\ref{fig:pdMM}. The right-hand side was generated using {\tt rand(n,1)}. }
 The  Block CD-pd method converged much faster on both problems. 
The lower condition number  ($\kappa_2=12$) of the \texttt{gr\_30\_30-rsa} problem resulted in fast convergence of all methods, see Figure~\ref{fig:gr_30_30-rsa}. While the high condition number  ($\kappa_2=4.3 \cdot 10^4$)  of the \texttt{bcsstk18} problem, resulted in a slow convergence for all methods, see Figure~\ref{fig:bcsstk18-rsa}.

\begin{figure}
    \centering
    \begin{subfigure}[t]{0.47\textwidth}
\includegraphics[width =  \textwidth, trim= 30 270 40 285, clip ]{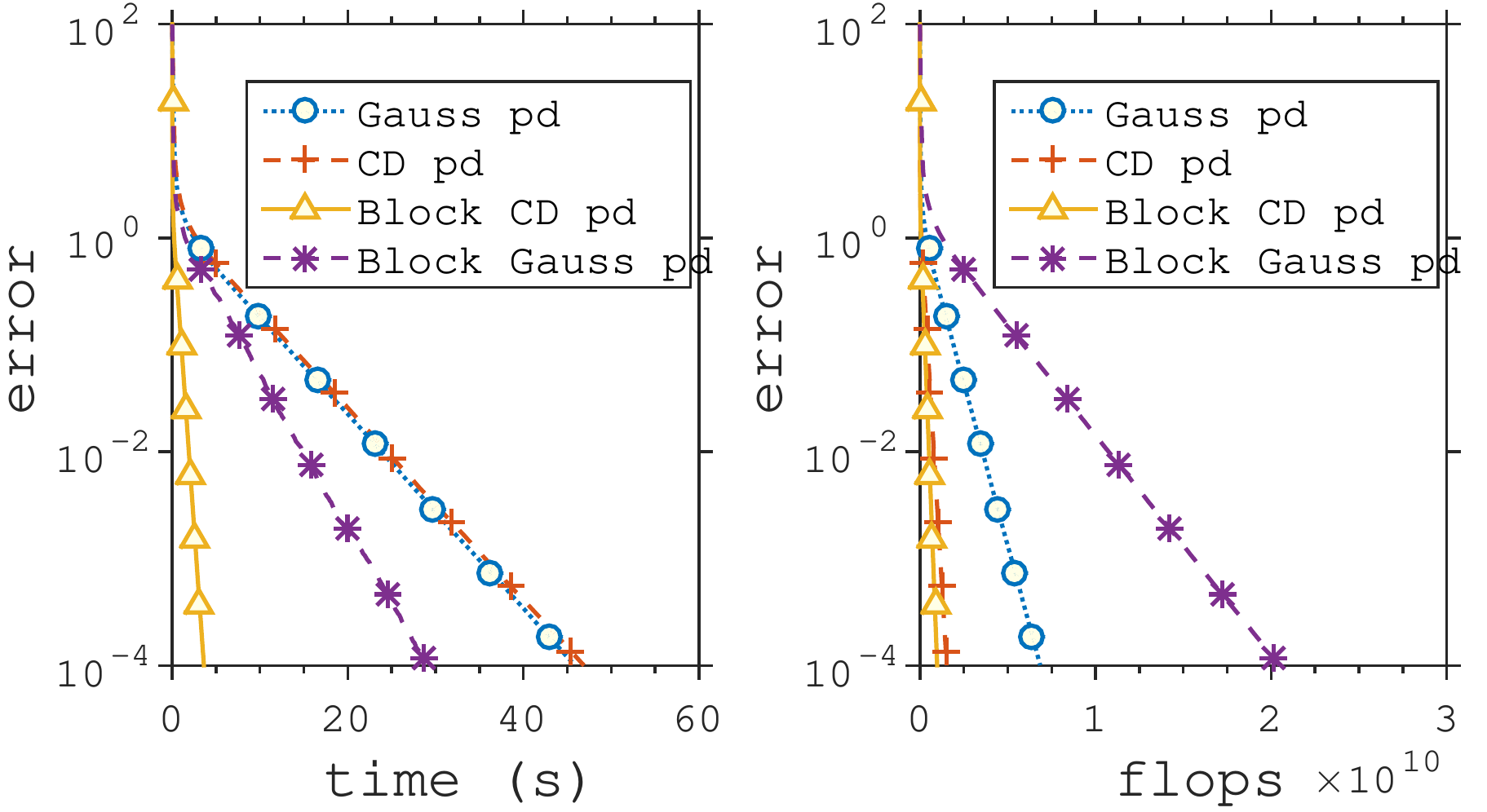}
        \caption{ \texttt{gr\_30\_30-rsa}}\label{fig:gr_30_30-rsa}
    \end{subfigure}%
    \hspace{0.05\textwidth}
    \begin{subfigure}[t]{0.47\textwidth}
\includegraphics[width =  \textwidth, trim= 30 270 40 285, clip ]{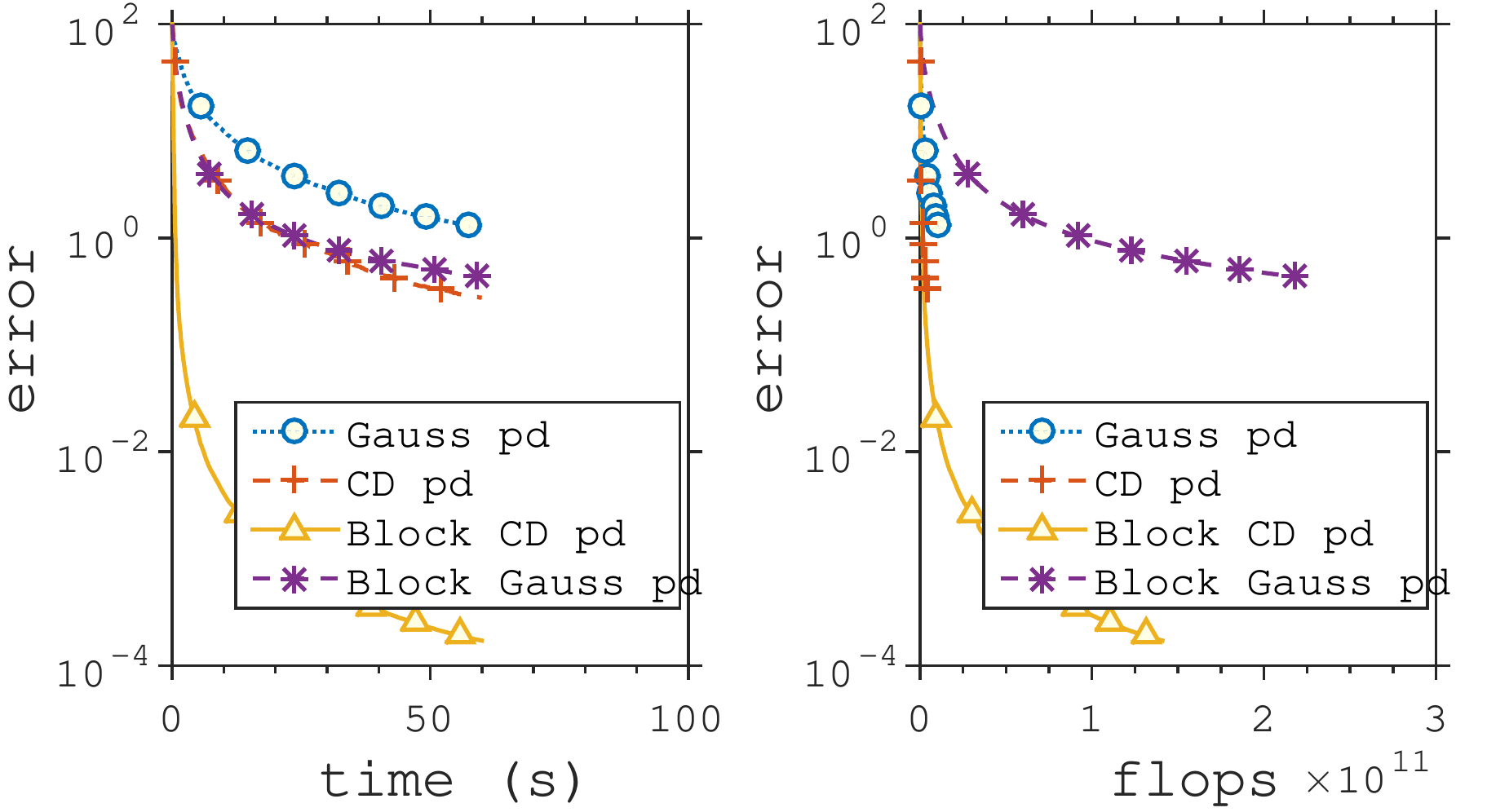}
        \caption{\texttt{bcsstk18}} \label{fig:bcsstk18-rsa}
    \end{subfigure}
    \caption{The performance of the Gauss-pd, CD-pd and the Block CD-pd methods on two linear systems from the MatrixMarket (a) \texttt{gr\_30\_30-rsa} with $n = 900$, $nnz = 4322$ ({\tt density}$=0.53\%$) and $\kappa_2 =12.$ (b) \texttt{bcsstk18} with $n = 11948$, $nnz=80519$ ({\tt density}$=0.1\%$) and  $\kappa_2 = 4.3 \cdot 10^{10} $.}\label{fig:pdMM}
\end{figure}

Despite the clear advantage of using a block variant, applying a block method that uses a direct solver can be infeasible on very ill-conditioned problems. As an example, 
 applying the Block CD-pd to the Hilbert system, and using MATLAB back-slash solver to solve the inner $q\times q$ systems, resulted in large numerical inaccuracies, and ultimately, prevented the method from converging. This occurred because the submatrices of the Hilbert matrix are also very ill-conditioned.

\subsection{Comparison between Optimized and Convenient probabilities}\label{sec:numopt}
 We compare the practical performance of using the convenient probabilities~\eqref{eq:convprob} against using the optimized probabilities by solving~\eqref{eq:optconv}. \rob{We solved~\eqref{eq:optconv} using the disciplined convex  programming solver \texttt{cvx}~\cite{cvx} for MATLAB.}

In Table~\ref{tab:CD-pd-opt} we compare the different convergence rates for the CD-pd method, where $\rho_c$ is the convenient convergence rate~\eqref{eq:rhoconv}, $\rho^*$ the optimized convergence rate, $(1-1/n)$ is the lower bound, and in the final ``optimized time(s)'' column the time taken to compute $\rho^*$. In Figure~\ref{fig:CD-pd-opt}, we compare the empirical convergence of the CD-pd method when using the convenient probabilities~\eqref{eq:convprob} and CD-pd-opt, the CD-pd method with the optimized probabilities, on four ridge regression problems and a uniform random matrix.
We ran each method for $60$ seconds. 

 In most cases using the optimized probabilities results in a much faster convergence, see Figures~\ref{fig:aloi-opt},~\ref{fig:liver-opt},~\ref{fig:mushrooms-opt} and~\ref{fig:uniform-opt}. \rob{In particular, the $7.401$} seconds spent calculating the optimal probabilities for \texttt{aloi} paid off with a convergence that was $55$ seconds faster. The \texttt{mushrooms} problem was insensitive to the choice of probabilities~\ref{fig:mushrooms-opt}. Finally despite $\rho^*$ being much less than $\rho_c$ on \texttt{covtype}, see Table~\ref{tab:CD-pd-opt}, using optimized probabilities  resulted in \rob{an initially slower method, though CD-pd-opt eventually catches up as CD-pd stagnates,} see Figure~\ref{fig:covtype-opt}.


\begin{table} \centering
\footnotesize
\begin{tabular}{c|lll|c} \hline
data set & $\rho_c$ &$\rho^* $ & $1-1/n$ & optimized time(s) \\ \hline
\texttt{rand}(50,50) & $1-2\cdot 10^{-6}$ & $1-3.05\cdot 10^{-6}$ & $1-2.10^{-2}$ & 1.076\\
 {\tt mushrooms-ridge} & $1-5.86\cdot 10^{-6}$ & $1-7.15\cdot 10^{-6}$ & $1-8.93\cdot 10^{-3}$ &  4.632\\
  {\tt aloi-ridge} & $1-2.17\cdot 10^{-7}$ & $1-1.26\cdot 10^{-4}$ & $1-7.81\cdot 10^{-3}$ &  7.401\\
  {\tt liver-disorders-ridge} & $1-5.16\cdot 10^{-4}$ & $1-8.25\cdot 10^{-3}$ & $1-1.67\cdot 10^{-1}$ &  0.413\\  
   {\tt covtype.binary-ridge} & $1-7.57\cdot 10^{-14}$ & $1-1.48\cdot 10^{-6}$ & $1-1.85\cdot 10^{-2}$ &  1.449\\  	
\end{tabular}
\caption{Optimizing the convergence rate for CD-pd.}
\label{tab:CD-pd-opt}
\end{table}

\begin{figure}[H]
    \centering
    \begin{subfigure}[t]{0.45\textwidth}
        \centering
\includegraphics[width =  \textwidth, trim=120 295 120 300, clip ]{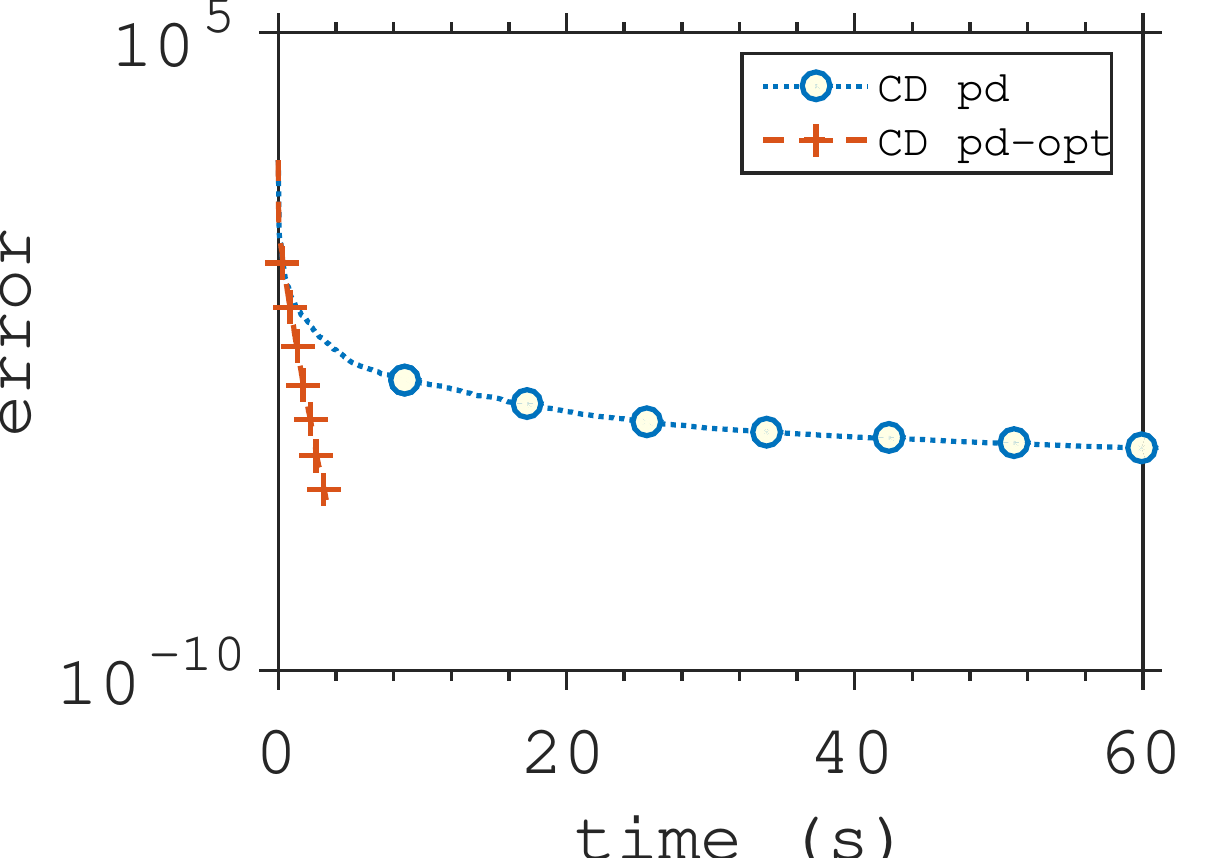}
        \caption{\texttt{aloi}}
        \label{fig:aloi-opt}
    \end{subfigure}%
  \hspace{0.02\textwidth}
    \begin{subfigure}[t]{0.45\textwidth}
        \centering
\includegraphics[width =  \textwidth, trim=120 295 120 300, clip ]{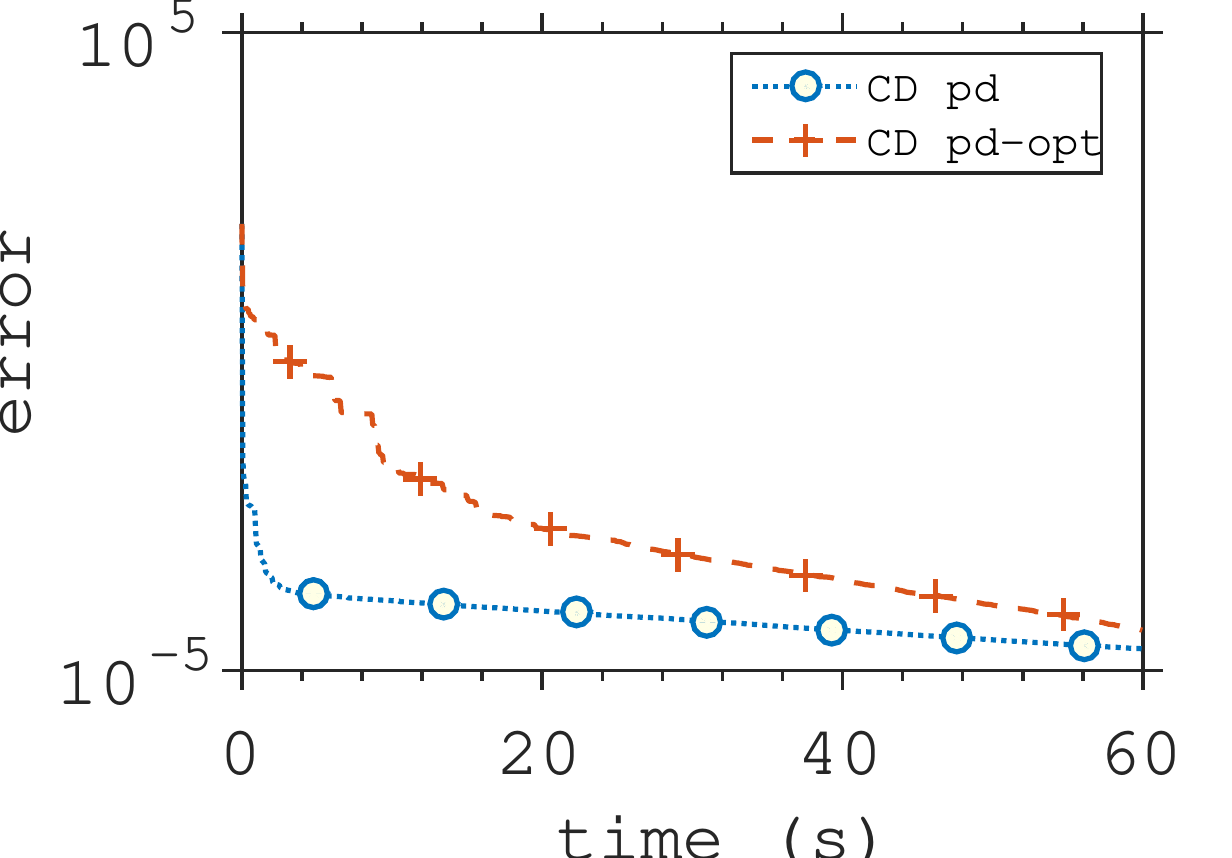}
        \caption{\texttt{covtype.libsvm.binary}}
        \label{fig:covtype-opt}
    \end{subfigure}
    \begin{subfigure}[t]{0.45\textwidth}
        \centering
\includegraphics[width =  \textwidth, trim=120 295 120 300, clip ]{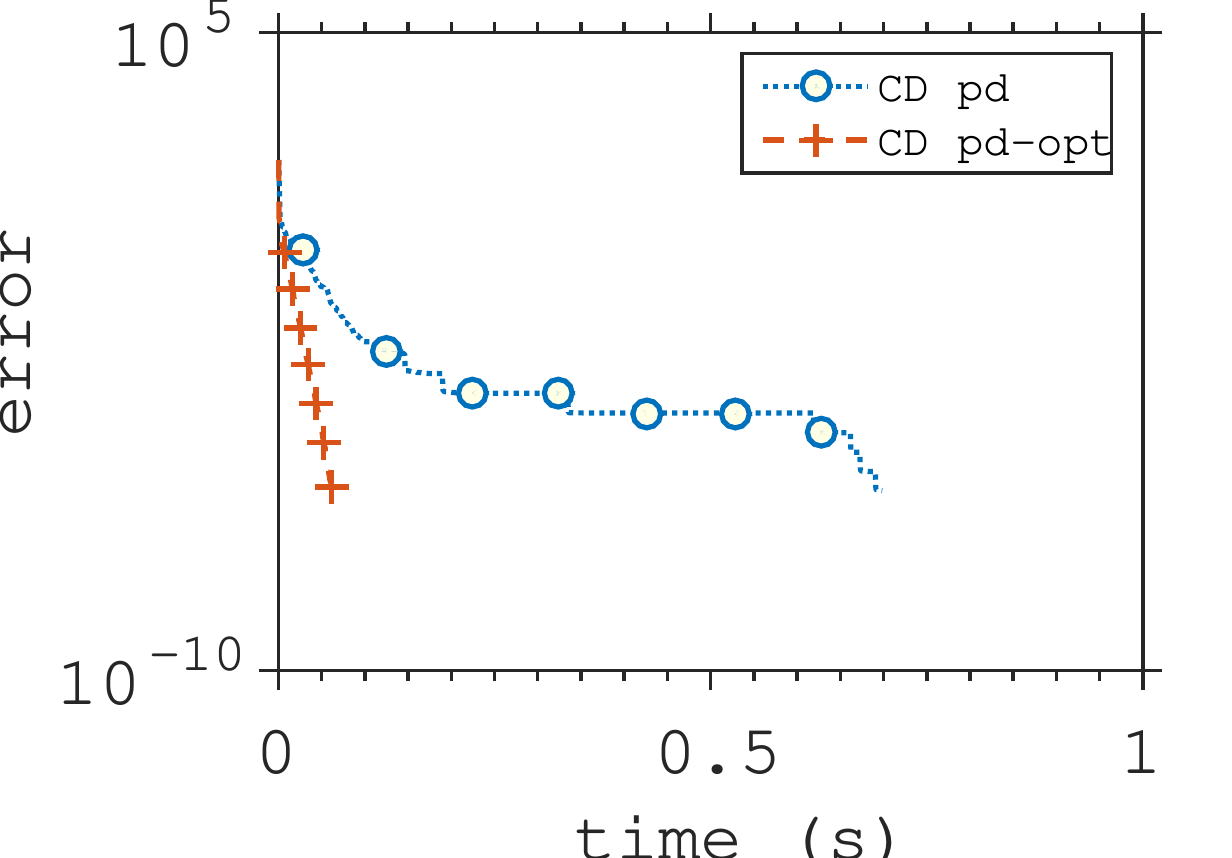}
        \caption{\texttt{liver-disorders-ridge}}
         \label{fig:liver-opt}
    \end{subfigure}
    \begin{subfigure}[t]{0.45\textwidth}
        \centering
\includegraphics[width =  \textwidth, trim=120 295 120 300, clip ]{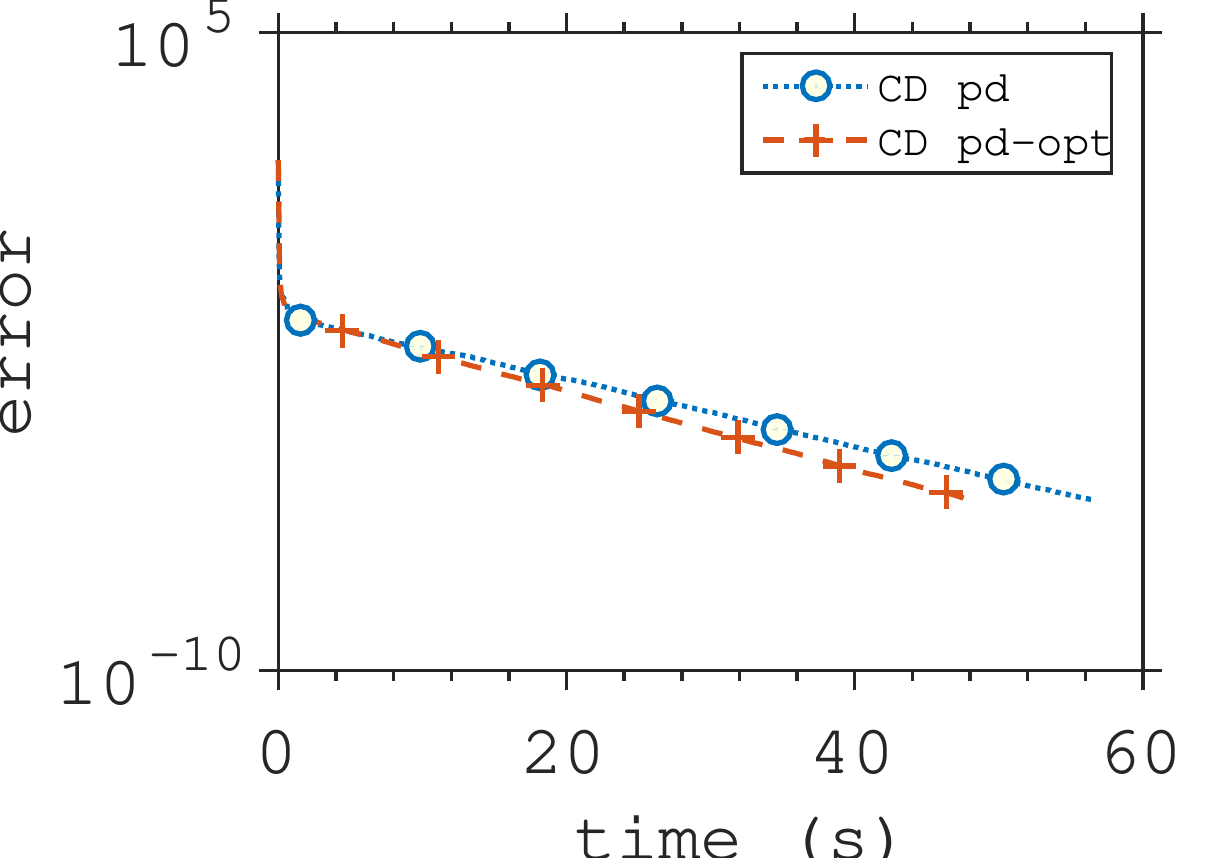}
        \caption{\texttt{mushrooms-ridge-opt}}
         \label{fig:mushrooms-opt}
    \end{subfigure} 
    \begin{subfigure}[t]{0.45\textwidth}
        \centering
\includegraphics[width =  \textwidth, trim=120 295 120 300, clip ]{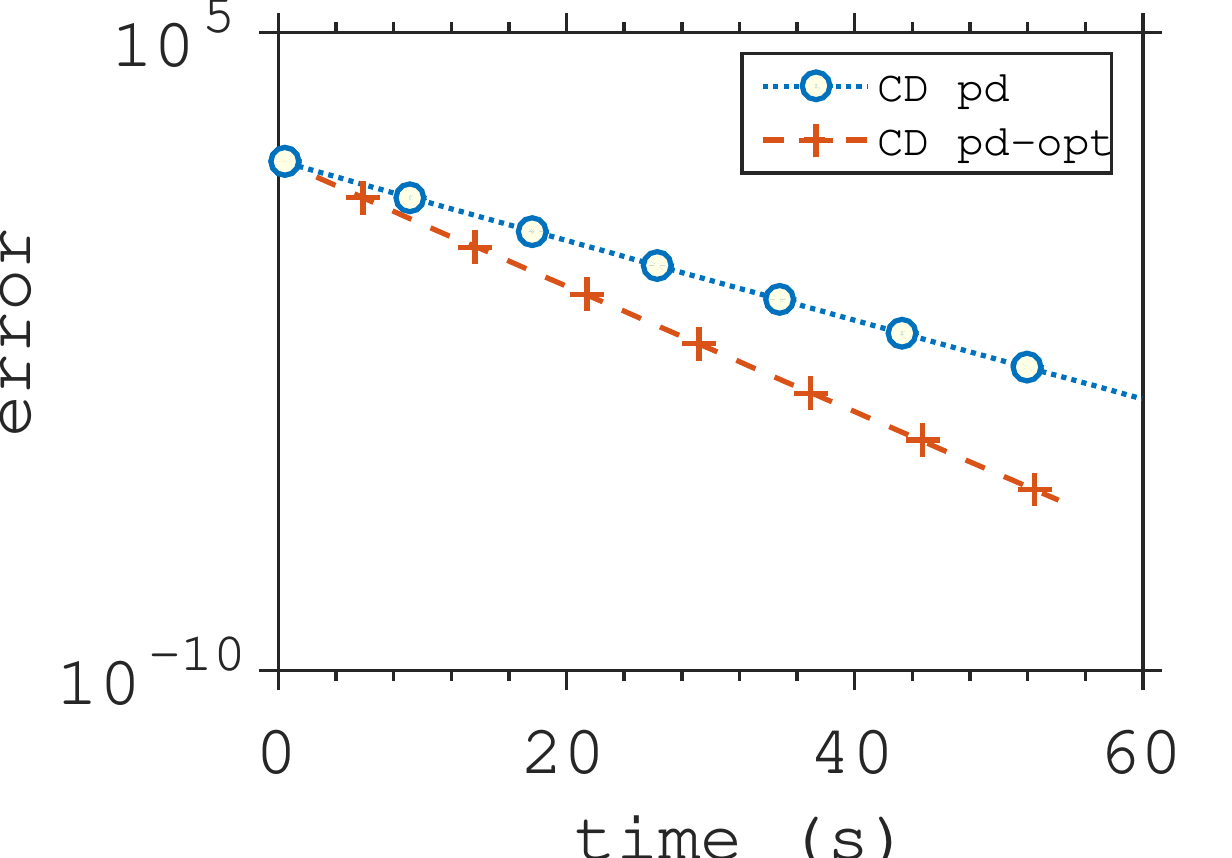}
        \caption{\texttt{uniform-random-50X50-opt}}
        \label{fig:uniform-opt}
    \end{subfigure}    
    \caption{The performance of CD-pd and optimized CD-pd methods on (a) \texttt{aloi}: $(m;n)=(108,000;128)$ (b) \texttt{covtype.binary}:  $(m;n)=(581,012; 54)$ (c) \texttt{liver-disorders}: $(m;n)=(345,6)$ (c)\texttt{mushrooms}: $(m;n) = (8124,112)$  (d) \texttt{uniform-random-50X50}.} \label{fig:CD-pd-opt}
\end{figure}

In Table~\ref{tab:RK-pd-opt} we compare the different convergence rates for the RK method. In Figure~\ref{fig:kaczmacz-opt}, we then compare the empirical convergence of the RK method when using the convenient probabilities~\eqref{eq:convprob} and RK-opt, the RK method with the optimized probabilities by solving~\eqref{eq:optconv}. The rates $\rho^*$ and $\rho_c$ for the  \texttt{rand}(500,100) problem are similar, and accordingly, both the convenient and optimized variant converge at a similar rate in practice, see Figure~\ref{fig:kaczmacz-opt}b. While the difference in the rates  $\rho^*$ and $\rho_c$ for the  {\tt liver-disorders} is more pronounced, and in this case, the $0.83$ seconds invested in obtaining the optimized probability distribution paid off in practice, as the optimized method converged $1.25$ seconds before the RK method with the convenient probability distribution, see Figure~\ref{fig:kaczmacz-opt}a.

\begin{table} \centering
\footnotesize
\begin{tabular}{c|lll|c} \hline
data set & $\rho_c$ &$\rho^* $ & $1-1/n$ & optimized time(s) \\ \hline
\texttt{rand}(500,100) & $1-3.37\cdot 10^{-3}$ & $1-4.27\cdot 10^{-3}$ & $1-1\cdot 10^{-2}$ & 33.121\\
  {\tt liver-disorders} & $1-5.16\cdot 10^{-4}$ & $1-4.04\cdot 10^{-3}$ & $1-1.67 \cdot 10^{-1}$ &  0.8316\\
\end{tabular}
\caption{Optimizing the convergence rate for randomized Kaczmarz.}
\label{tab:RK-pd-opt}
\end{table}

\begin{figure}
    \centering
    \begin{subfigure}[t]{0.45\textwidth}
        \centering
\includegraphics[width =  \textwidth, trim=120 295 120 300, clip ]{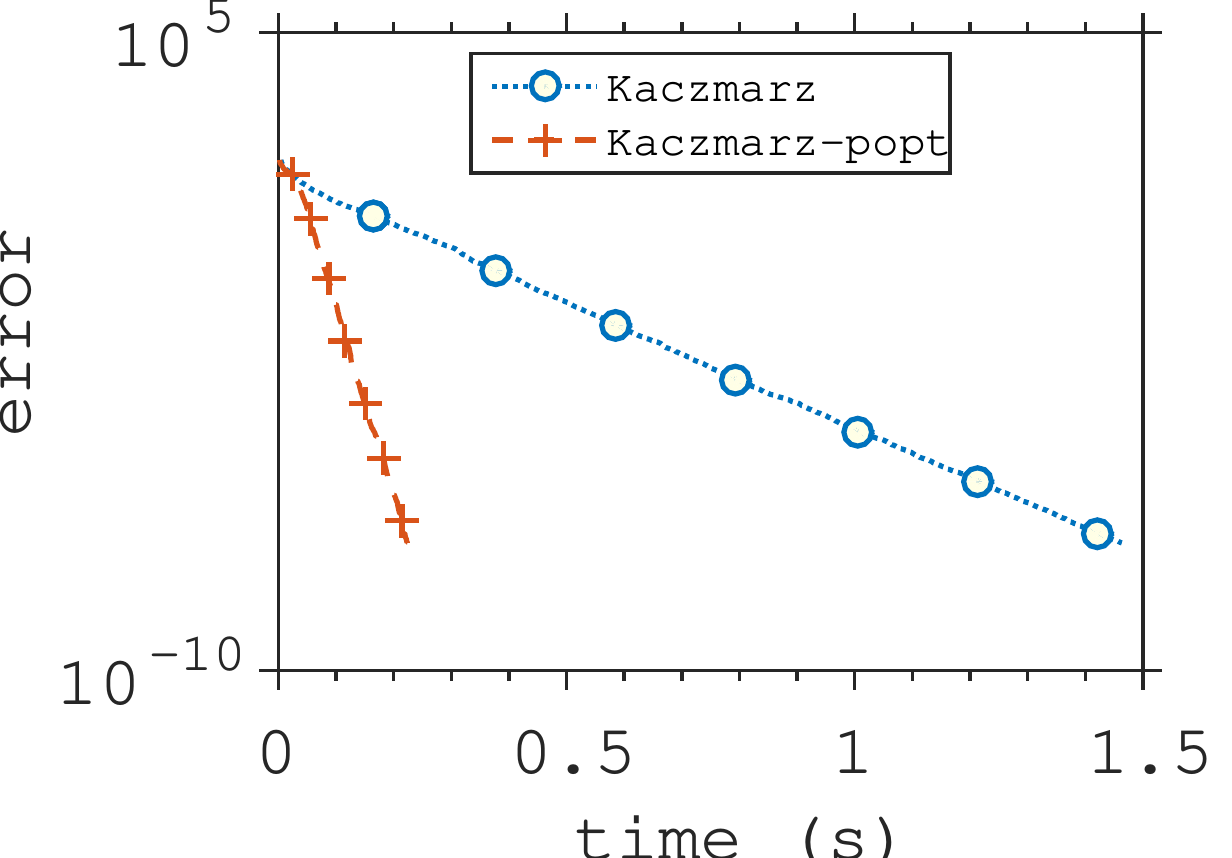}
        \caption{\texttt{liver-disorders-popt-k}}
    \end{subfigure}%
  \hspace{0.02\textwidth}
    \begin{subfigure}[t]{0.45\textwidth}
        \centering
\includegraphics[width =  \textwidth, trim=120 295 120 300, clip ]{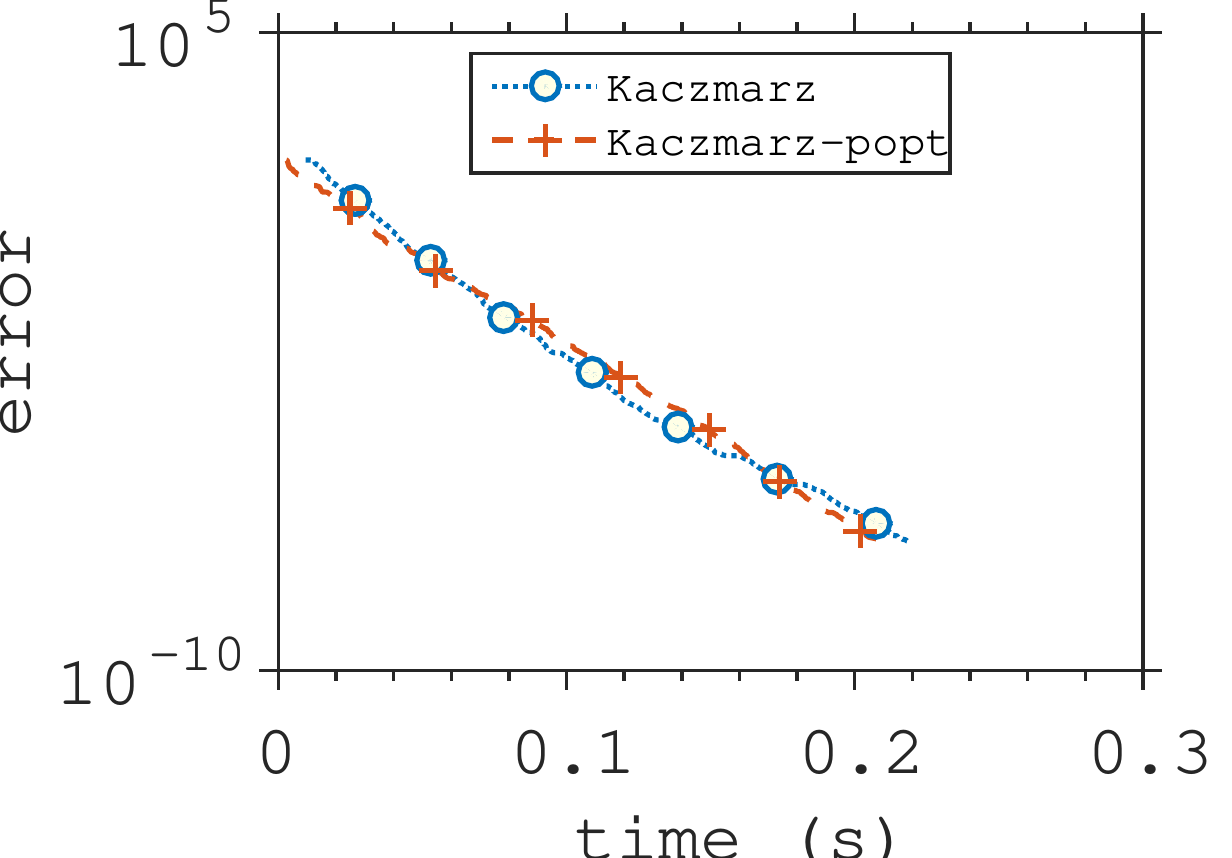}
        \caption{\texttt{rand}(500,100)}
    \end{subfigure}
    \caption{The performance of Kaczmarz and optimized Kaczmarz methods on  (a) \texttt{liver-disorders}: $(m;n)=(345,6)$  (b) \texttt{rand}(500,100)}
\label{fig:kaczmacz-opt}
\end{figure} 

We conclude from these tests that the choice of the probability distribution can greatly affect the performance of the method. Hence, it is worthwhile to develop approximate solutions to~\eqref{eq:opt_sampling}.

\section{Conclusion} \label{sec:conclusion}

We present a unifying framework for the randomized Kaczmarz method, randomized Newton method, randomized coordinate descent method and  random Gaussian pursuit. Not only can we recover these methods by selecting appropriately the parameters $S$ and $B$, but also, we can analyse them and their block variants through a single Theorem~\ref{theo:Enormconv}. Furthermore, we obtain a new lower bound for all these methods in Theorem~\ref{theo:normEconv}, and in the discrete case, recover all known convergence rates expressed in terms of the scaled condition number in Theorem~\ref{theo:convsingleS}. 

The Theorem~\ref{theo:convsingleS} also suggests a preconditioning strategy. Developing preconditioning methods are important for reaching a higher precision solution on ill-conditioned problems. For as we have seen in the numerical experiments, the randomized methods struggle to bring the solution within $10^{-2}$ relative error when the matrix is ill-conditioned.

This is also a framework on which randomized methods for linear systems can be designed. As an example, we have designed a new block variant of RK, a new Gaussian Kaczmarz method \rob{and a new Gaussian block method for positive definite systems}. Furthermore, the flexibility of our framework and the general convergence Theorems~\ref{theo:Enormconv} and~\ref{theo:normEconv} allows one to tailor the probability distribution of $S$ to a particular problem class. For instance, other continuous distributions such uniform, or other discrete distributions such Poisson might be more suited to a particular class of problems. 

Numeric tests reveal that the new Gaussian methods designed for overdetermined systems are competitive on sparse problems, as compared to the Kaczmarz and CD-LS methods. The Gauss-pd also proved competitive as compared to  CD-pd on all tests. Though, when applicable, the combined efficiency of using a direct solver and an iterative procedure, such as in  Block CD-pd method, proved the most efficient.

The work opens up many possible future venues of research.  Including investigating accelerated convergence rates through preconditioning strategies based on Theorem~\ref{theo:convsingleS} or by obtaining approximate optimized probability distributions~\eqref{eq:optconv}.  

\subsubsection*{Acknowledgments}
The authors would like to thank Prof. Sandy Davie for useful discussions relating to  Lemma~\ref{lem:2Dgausscov}\rob{, and Prof. Joel Tropp for invaluable suggestions regarding Lemma~\ref{lem:gaussdiag}.}

{\small
\printbibliography
}

\appendix

\section{A Bound on the Expected Gaussian Projection Matrix}
\rob{
\begin{lemma} \label{lem:gaussdiag}
Let $D \in \R^{n\times n}$ be a positive definite diagonal matrix, $U \in \R^{n\times n}$ an orthogonal matrix and $\Omega  =UDU^T$. If $u \sim N(0,D)$ and $\xi \sim N(0,\Omega)$  then 
\begin{equation}\label{eq:gaussdiag}\E{\frac{\xi \xi^T}{\xi^T\xi}} =U\E{\frac{u u^T}{u^Tu}}U^T ,\end{equation}
and
\begin{equation}\label{eq:gaussupper}
\E{\frac{\xi \xi^T}{\xi^T\xi}} \succeq \frac{2}{\pi}\frac{\Omega}{\Tr{\Omega}}.
\end{equation}
\end{lemma}
\begin{proof}
Let us write $S(\xi)$ for the random vector $\xi/\|\xi\|_2$ (if $\xi=0$, we set $S(\xi)=0$).  Using this notation, we can write \[ \E{ \xi (\xi^T \xi)^{-1} \xi^T  }  = \E{S(\xi)(S(\xi))^T} = \COV{S(\xi)},\]
where the last identity follows since $\E{S(\xi)}=0$, which in turn holds as the Gaussian distribution is centrally symmetric. As $\xi = U u$,  note that  
\[S(u) = \frac{U^T \xi}{\|U^T \xi\|_2} = \frac{U^T \xi}{\|\xi\|_2} = U^T S(\xi).\]
Left multiplying both sides by $U$ we obtain
$U S(u) = S(\xi)$, from which we obtain
\[\COV{S(\xi)} = U \COV{ S(u) } U^T,  
\]
which is equivalent to~\eqref{eq:gaussdiag}.

   To prove \eqref{eq:gaussupper}, note first that  $M \eqdef \E{u u^T/u^Tu}$ is a diagonal matrix.  One can verify this by direct calculation (informally, this holds because the entries of $u$ are independent and centrally symmetric). The $i$th diagonal entry is given by
\[M_{ii} = \E{\frac{u_{i}^2}{\sum_{j=1}^n u_j^2}}.\]

%
As the map $(x,y) \rightarrow x^2/y$ is convex on the positive orthant, we can apply Jensen's inequality, which gives
\[\E{\frac{u_{i}^2}{\sum_{j=1}^n u_j^2}} \geq \frac{\left(\E{|u_{i}|}\right)^2}{\sum_{j=1}^n \E{u_j^2}}
= \frac{2}{\pi}\frac{D_{ii}}{\Tr{D}},\]
which concludes the proof.   
\end{proof}
}


\section{Expected Gaussian Projection Matrix in 2D}\rob{
\begin{lemma} \label{lem:2Dgausscov}
Let $\xi \sim N(0,\Omega)$ and $\Omega \in \R^{2\times 2}$ be a positive definite matrix, then
\begin{equation}\label{eq:2Dgausscov}\E{\frac{\xi \xi^T}{\xi^T\xi}} = \frac{\Omega^{1/2}}{\Tr{\Omega^{1/2}}}.\end{equation}
\end{lemma}

\begin{proof}
Let $\Sigma = U D U^T$ and $u \sim N(0,D).$ Given~\eqref{eq:gaussdiag} it suffices to show that 
\begin{equation}\label{eq:h98hs98hs9ss}\COV{S(u)} = \frac{D^{1/2}}{\Tr{D^{1/2}}},\end{equation}
which we will now prove.

Let $\sigma_x^2$ and $\sigma_y^2$ be the two diagonal elements of $D.$ First, suppose that $\sigma_x = \sigma_y.$ Then $u = \sigma_x \eta$ where $\eta \sim N(0,I)$ and
\[\E{\frac{u u^T}{u^Tu}} = \frac{\sigma_x^2}{\sigma_x^2}\E{\frac{\eta \eta^T}{\eta^T\eta}} = \frac{1}{n}I = \frac{D^{1/2}}{\Tr{D^{1/2}}}. \]
Now suppose that $\sigma_x \neq \sigma_y.$ We calculate the diagonal terms of the covariance matrix by integrating
 \[\E{\frac{u_1^2}{u_1^2+u_2^2}}=\frac{1}{2\pi \sigma_x\sigma_y}\int_{\R^2} \frac{x^2}{x^2+y^2} e^{-\frac{1}{2}\left( x^2/\sigma_x^2 +y^2/\sigma_y^2 \right) } dxdy. \]

 Using polar coordinates $x= R \cos(\theta)$ and $y =R \sin(\theta)$ we have
\begin{equation}\int_{\R^2} \frac{x^2}{x^2+y^2} e^{-\frac{1}{2}\left( x^2/\sigma_x^2 +y^2/\sigma_y^2 \right) } dxdy =
\int_0^{2\pi}\int_{0}^{\infty} R\cos^2(\theta) e^{-\frac{R^2}{2} C(\theta) } dRd\theta,\label{eq:ondstep}
\end{equation}
where $C(\theta) \eqdef \left( \cos(\theta)^2/\sigma_x^2 +\sin(\theta)^2/\sigma_y^2 \right). $ Note that
\begin{equation}\label{eq:Rintfirst}
\int_{0}^{\infty} R e^{-\frac{C(\theta)R^2}{2} } dR = \left.-\frac{1}{C(\theta)}e^{-\frac{C(\theta)R^2}{2}} \right|_{0}^{\infty} =\frac{1}{C(\theta)}. 
\end{equation}
This applied in~\eqref{eq:ondstep} gives
\begin{align*} \E{\frac{u_1^2}{u_1^2+u_2^2}} &=\frac{1}{2\pi \sigma_x \sigma_y}  \int_0^{2\pi}\frac{\cos^2(\theta)}{ \cos(\theta)^2/\sigma_x^2 +\sin(\theta)^2/\sigma_y^2  }d\theta = \frac{b}{\pi}  \int_0^{\pi}\frac{\cos^2(\theta)}{ \cos^2(\theta) +b^2\sin^2(\theta)}d\theta, \end{align*}
where $b = \sigma_x/\sigma_y.$
Multiplying the numerator and denominator of the integrand by $\sec^4(x)$ gives the integral
\[\E{\frac{u_1^2}{u_1^2+u_2^2}} = \frac{b}{\pi} \int_0^{\pi}\frac{\sec^2(\theta)}{ \sec(\theta)^2\left( 1 +b^2\tan^2(\theta) \right) }d\theta.\]
Substituting $v=\tan(\theta)$ so that $v^2 +1=\sec^2(\theta)$, $dv =\sec^2(\theta)d \theta$ and using the partial fractions
\[\frac{1}{(v^2+1)\left( 1 +b^2v^2 \right)} = \frac{1}{1-b^2}\left( \frac{1}{v^2+1} -\frac{b^2}{b^2v^2+1} \right), \]
gives the integral
\begin{align}
\int \frac{dv}{ (v^2+1)\left( 1 +b^2 v^2 \right) } &=
\frac{1}{1-b^2}\left(\arctan(v)-b\arctan(b v) \right) \nonumber\\
&= \frac{1}{1-b^2}\left(\theta-b\arctan(b \tan(\theta))\right). \label{eq:A121}
\end{align} 
To apply the limits of integration, we must take care because of the singularity at $\theta=\pi/2$. For this, consider the limits 
\[\lim_{\theta \rightarrow (\pi/2)^-} \arctan(b\tan(\theta)) = \frac{\pi}{2}, \qquad \lim_{\theta \rightarrow (\pi/2)^+} \arctan(b\tan(\theta)) = -\frac{\pi}{2}.\]
Using this to evaluate~\eqref{eq:A121} on the limits of the interval $[0,\, \pi/2]$  gives 
\[\lim_{t \rightarrow (\pi/2)^-}\left.\frac{1}{1-b^2}\left(\theta-b\arctan(b \tan(\theta))\right)\right|_0^{t}=
 \frac{1}{1-b^2}\frac{\pi}{2}(1-b) = \frac{\pi}{2(1+b)}.\]
 Applying a similar argument for calculating the limits from  $\pi/2^+$ to $\pi$, we find
\[ \E{\frac{u_1^2}{u_1^2+u_2^2}} =\frac{2b}{\pi}  \frac{ \pi}{2(1+b)} =\frac{\sigma_x}{\sigma_y+\sigma_x}.\]
Repeating the same steps with $x$ swapped for $y$ we obtain the other diagonal element, which concludes the proof of~\eqref{eq:h98hs98hs9ss}.
\end{proof}}

\end{document}